\documentclass[10pt]{article}
\usepackage{amsmath, amssymb, amsthm}
\usepackage{appendix}
\usepackage{multicol} 
\usepackage{makeidx}
\usepackage{cite}
\usepackage{captdef}
\makeindex
\usepackage{paralist}
\usepackage{psfrag}
\usepackage{graphicx}
\usepackage{graphics}
\usepackage{color}

\newcommand{\R}{\mathbb{R}}

\newtheorem{theo}{Theorem}[section]
\newtheorem{coro}[theo]{Corollary}
\newtheorem{lemma}[theo]{Lemma}
\newtheorem{prop}[theo]{Proposition}
\theoremstyle{definition}

\newtheorem{remark}[theo]{Remark}

\DeclareMathOperator{\dist}{dist}

\newcommand{\cI}{{\mathcal I}}

\newcommand{\cL}{{\mathcal L}}   
\newcommand{\cM}{{\mathcal M}}   
\newcommand{\cN}{{\mathcal N}}

\newcommand{\cU}{{\mathcal U}}

\newcommand{\eps}{\varepsilon}

\renewcommand{\epsilon}{\varepsilon}

\newcommand{\Sn}{{\mathbb S}^{N-1}}

\numberwithin{equation}{section}

\title{On the asymptotic shape of solutions to Neumann problems 
  for non-cooperative parabolic systems}
\author{Alberto Salda\~{n}a\footnote{Institut f\"{u}r Mathematik, Johann Wolfgang Goethe-Universit\"{a}t Frankfurt, Robert-Mayer-Str. 10, D-60054 Frankfurt, saldana@math.uni-frankfurt.de.}  \ \ \  \& \ \ \ Tobias Weth\footnote{Institut f\"{u}r Mathematik, Johann Wolfgang Goethe-Universit\"{a}t Frankfurt, Robert-Mayer-Str. 10, D-60054 Frankfurt, weth@math.uni-frankfurt.de.} }
\date{}

\begin{document}
\maketitle

\begin{abstract}
We consider a class of nonautonomous parabolic 
competition-diffusion systems on bounded radial domains under Neumann boundary conditions. We show
that, if the initial profiles satisfy a reflection inequality with
respect to a hyperplane, then global positive solutions are
asymptotically (in time) foliated Schwarz symmetric with respect to
antipodal points. Additionally, a related result for (positive and
sign changing solutions) of a scalar equation with Neumann boundary
conditions is given. The asymptotic shape of solutions to cooperative systems is also discussed.
\end{abstract}
{\footnotesize
\begin{center}
\end{center}
\begin{center}
\end{center}
}

\section{Introduction}

The focus of the present paper is the asymptotic shape of positive
global solutions of parabolic systems with competition on bounded radial
domains with Neumann boundary conditions. The problem which mainly
motivates our
study is the following Lotka-Volterra System of two equations:
\begin{equation}\label{Lotka:Volterra:system}
\begin{aligned}
(u_1)_t-\mu_1\Delta u_1 &= a_1(t) u_1 - b_1(t) u_1^2 -\alpha_1(t) u_1u_2,&&\quad x \in B,\ t>0,\\
(u_2)_t-\mu_2\Delta u_2 &= a_2(t) u_2 - b_2(t) u_2^2 -\alpha_2(t) u_1u_2,&&\quad x\in B,\ t>0,\\
\frac{\partial u_i}{\partial\nu}&=0 &&\quad \text{on $\partial B
  \times (0,\infty)$,}\\
u_i(x,0)&=u_{0,i}(x)\geq 0&&\quad \text{for $x\in B,$ $i=1,2$.}
\end{aligned}
\end{equation}
Here and in the remainder of the paper, $B$ denotes a ball or an annulus in
$\mathbb R^N$ with $N\geq 2$, and $\nu$ denotes the unit outer normal
on $\partial B$. Moreover, $\mu_1$ and $\mu_2$ are positive
constants and 
\begin{equation}\label{Lotka:coefficients}
  \begin{aligned}
  &\text{$a_i,b_i, \alpha_i \in L^\infty((0,\infty))$ satisfy}\\   
  &\text{$a_i(t),b_i(t) \geq 0$ for $t>0$ \ \ and\ \  $\inf_{t>0} \alpha_i(t)>0$ for }i=1,2.
  \end{aligned}
\end{equation}
The Lotka-Volterra System is commonly used to
model the competition between two different species, and the
coefficients $\mu_i, a_i, b_i, \alpha_i$ represent diffusion, birth,
saturation, and competition rates respectively (see
\cite{holmes}). In the literature, the system is mostly considered
with constant coefficients for matters of simplicity, whereas
it is more natural to assume time-dependence as e.g. in
\cite{cantrell,langa,Mierczynski} in order to model
the effect of different time periods on the birth rates, the movement, or the aggressiveness of the species. Even in the case of constant coefficients,
the possible dynamics of the system has a very rich structure and
depend strongly on relationships
between these constants, see e.g.
\cite{cantrell,crooks,dancer,lou,dancer:zhang}. In the case of
time-dependent coefficients, a full understanding of the asymptotic
dynamics is out of reach, but one can still guess that the shape
of the underlying domain has some effect on the shape of solutions for
large positive times. In the present paper we study this question on
a radial bounded domain $B$. More precisely, for a solution
$u=(u_1,u_2)$ of (\ref{Lotka:Volterra:system}), we study symmetry and
monotonicity properties of elements in the associated omega limit set,
which is defined as 
\begin{align*}
\omega(u)&= \omega(u_1,u_2):=\{(z_1,z_2)\in C(\overline{B})\times
C(\overline{B})\mid \\&
\max_{i=1,2}\lim_{n\to\infty}\|u_i(\cdot,t_n)-z_i\|_{L^\infty(B)}=0
\text{ for some sequence } t_n\to\infty\}.\nonumber
\end{align*}
For global solutions which are uniformly bounded and have
equicontinuous semiorbits $\{u_i(\cdot,t):t\geq 1\}$, the set
$\omega(u)$ is nonempty, compact, and connected. The
equicontinuity can be obtained under mild boundedness and regularity
assumptions on the equation and using boundary and interior H\"{o}lder
estimates (see Lemma \ref{regularity} below). To present our results we need to introduce some notation. Let $\Sn=\{x\in\mathbb R^N: |x|=1\}$ be the unit sphere in
$\R^N,$ $N\geq 2$.  For a vector $e\in \Sn$, 
we consider the hyperplane $H(e):=\{x\in \mathbb R^N: x\cdot e=0\}$
and the half domain $B(e):=\{x\in B: x\cdot e>0\}.$  We write also
$\sigma_e: \overline{B}\to \overline{B}$ to denote reflection with respect to $H(e),$
i.e. $\sigma_e(x):=x-2(x\cdot e)e$ for each $x\in B.$ Following
\cite{smets}, we say that a
function $u\in C(B)$ is \textit{foliated Schwarz symmetric with
respect to some unit vector $p\in \Sn$} if $u$ is axially symmetric with
respect to the axis $\mathbb R p$ and nonincreasing in the polar angle
$\theta:= \operatorname{arccos}(\frac{x}{|x|}\cdot p)\in [0,\pi].$ 
We refer the reader to the survey article \cite{wethsurvey} for a
broad discussion of symmetry properties of this type. Our main result concerning (\ref{Lotka:Volterra:system}) is the
following. 

\begin{theo}\label{corollary:Lotka:Volterra}
Suppose that 
\eqref{Lotka:coefficients} holds, and let $u=(u_1,u_2)$ be a classical
solution of \eqref{Lotka:Volterra:system} such that
$\|u_i\|_{L^\infty(B \times (0,\infty))} < \infty$ for
$i=1,2$. Moreover, assume that 
$$
\leqno{\rm (h0)}\quad u_{0,1} \geq u_{0,1} \circ \sigma_e, \:u_{0,2} \leq
u_{0,2}\circ \sigma_e \quad \text{in $B(e)$}\ \
\left \{
\begin{aligned}
&\text{for some $e\in \Sn$ with}\\
&\text{$u_{0,i}\not\equiv u_{0,i}\circ \sigma_e$ for $i=1,2$.}  
\end{aligned}
\right.
$$
Then there is some $p\in \Sn$ such that every $(z_1,z_2)\in \omega(u)$ has
the property that $z_1$ is foliated Schwarz symmetric with respect to $p$ and $z_2$ is foliated Schwarz symmetric with respect to $-p.$ 
\end{theo}

This theorem is a direct consequence of a more general result which we
state in Theorem~\ref{main:theorem:neumann} below. Note that the
inequality condition (h0) does not seem very strong, but it is a key
assumption in order to obtain the result. Indeed, for general positive
initial data, foliated Schwarz symmetry cannot be expected, as one may
see already by looking at equilibria (i.e., stationary solutions) in special cases. Consider e.g. the elliptic system  
\begin{equation}\label{Lotka:Volterra:system-elliptic}
\begin{aligned}
-\Delta u_1 &= \lambda u_1 - u_1u_2&&\qquad \text{ in } B,\\
-\Delta u_2 &= \lambda u_2 -  u_1u_2&&\qquad \text{ in } B,\\
\partial_\nu u_1&=\partial_\nu u_2=0&&\qquad  \text{ on $\partial B.$}
\end{aligned}
\end{equation}
Using bifurcation theory, one can detect values $\lambda>0$ and $\eps>0$ such that (\ref{Lotka:Volterra:system-elliptic}) admits positive solutions in the annulus $B= \{x \in \R^2\::\: 1-\eps < |x|<1\}$ such that the angular derivatives of $u_1,u_2$ change sign multiple times and therefore 
neither of the components is foliated Schwarz symmetric, see Theorem~\ref{thm:local:maxima:system}
below in the appendix. Theorem~\ref{corollary:Lotka:Volterra} is
somewhat related to our previous work \cite{saldana:weth}
on scalar nonlinear parabolic equations under Dirichlet boundary
conditions. The main idea of both \cite{saldana:weth} and the present
paper is to obtain the symmetry of elements in the omega limit set by
a rotating plane argument. However, different tools are required to set up the
method under Neumann boundary conditions. In particular, the main
result of \cite{saldana:weth} does 
not extend in a straightforward way to the scalar nonlinear Neumann problem  
\begin{equation}\label{model:scalar}
\begin{aligned}
 u_t-\mu(|x|,t)\Delta u &=f(t,|x|,u)&&\qquad  \text{in $B \times (0,\infty)$,}\\
\partial_\nu u&=0&&\qquad  \text{on $\partial B \times (0,\infty)$,}\\
 u(x,0)&=u_0(x)&& \qquad \text{for $x\in B$.}
\end{aligned}
\end{equation}
It therefore seems appropriate to include a symmetry result for (positive and sign changing)
solutions of (\ref{model:scalar}) in the present paper. This result will be easier to prove
than Theorem~\ref{corollary:Lotka:Volterra}. We need the following hypotheses on the nonlinearity $f$ and the
diffusion coefficient $\mu$. In the following, we put $I_B:=\{|x| : x\in \overline{B}\}$.    
\begin{enumerate}
\item[(H1)] The nonlinearity $f:[0,\infty)\times I_B \times \R\to\R,$ $(t,r,u)\mapsto f(t,r,u)$ is
  locally Lipschitz continuous in $u$ uniformly in $r$
  and $t,$ i.e.,  
$$
\sup_{\genfrac{}{}{0pt}{}{\scriptstyle{r\in I_B,
      t>0,}}{\scriptstyle{u,\bar u\in K, u\neq\bar
      u}}}\!\!\!\frac{|f(t,r,u)-f(t,r,\bar u)|}{|u-\bar u|}< \infty\ \ \  \text{for any compact subset $K\subset [0,\infty)$.}
$$
\item[(H2)] $\sup\limits_{r\in I_B, t>0} |f(t,r,0)|<\infty.$
\item[(H3)] $\mu\in C^1(I_B\times(0,\infty))$ and there are constants
  $\mu^* \ge \mu_* >0$ such that
  $\|\mu_i\|_{C^1(I_B\times(0,\infty))}\le \mu^*$ and  $\mu_i(r,t)\geq \mu_*$ for all $r\in I_B,$ $t>0.$
\end{enumerate}

The following is our main result on (\ref{model:scalar}).
\begin{theo}\label{main:theorem:scalar}
Assume that (H1)-(H3) are satisfied, and let $u\in C^{2,1}(\overline{B}\times(0,\infty))\cap
C(\overline{B}\times[0,\infty))$ be a classical solution of
\eqref{model:scalar} such that  
\begin{equation}
\label{eq:15}
\|u\|_{L^\infty(B\times(0,\infty))}<\infty.
\end{equation}
Suppose furthermore that 
\begin{enumerate}
\item [(H4)] there is $e\in \Sn$ such that $u_0 \geq u_0 \circ
  \sigma_e$ in $B(e)$ and $u_0\not\equiv u_0\circ
  \sigma_e$. 
\end{enumerate}
Then, there is some $p\in \Sn$ such that every element of the
omega limit set 
$$
\omega(u):= \{z \in C(\overline{B})\mid \lim_{n\to\infty}\|u(\cdot,t_n)-z\|_{L^\infty(B)}=0
\text{ for some sequence } t_n\to\infty\}
$$
is foliated Schwarz symmetric with respect to $p$.
\end{theo}

We now turn to a general class of two-component nonlinear
competitive systems which includes (\ref{Lotka:Volterra:system}). More
precisely, we consider, for $i=1,2$, 
\begin{equation}\label{model:competitive:neumann}
\begin{aligned}
 (u_i)_t-\mu_i(|x|,t)\Delta u_i &=f_i(t,|x|,u_i)-\alpha_i(|x|,t)
 u_1u_2, &&\; \text{$x\in B,\,t>0$,}\\
\partial_\nu u_i&=0,&&\; \text{$x\in \partial B,\ t>0$,}\\
 u_i(x,0)&=u_{0,i}(x) \ge 0 && \; \text{for $x\in B$.}
\end{aligned}
\end{equation}
On the data, we assume the following. 
\begin{itemize}
\item[(h1)] For $i=1,2,$ the function $f_i:[0,\infty)\times I_B\times [0,\infty)\to\R,$ $(t,r,v)\mapsto f(t,r,v)$ is locally Lipschitz continuous in $v$ uniformly with respect to $r$ and $t,$ i.e.  
$$
\sup_{\genfrac{}{}{0pt}{}{\scriptstyle{r\in I_B,
      t>0,}}{\scriptstyle{v,\bar v\in K, v\neq\bar
      v}}}\!\!\!\frac{|f_i(t,r,v)-f_i(t,r,\bar v)|}{|v-\bar v|}<
\infty\ \ \  \text{for any compact subset $K\subset [0,\infty)$}.
$$
Moreover $f_1(t,r,0)=f_2(t,r,0)=0$ for all $r\in I_B, t>0.$
\item[(h2)] $\mu_i\in C^{2,1}(I_B\times(0,\infty))$ and there are
  constants $\mu^* \ge \mu_*>0$ such that
  $\|\mu_i\|_{C^{2,1}(I_B\times(0,\infty))}\le \mu^*$ and  $\mu_i(r,t)\geq
  \mu_*$ for all $r\in I_B,$ $t>0,$ and $i=1,2.$
\item[(h3)] $\alpha_i\in L^\infty(I_B\times(0,\infty))$ and there are
 constants  $\alpha^* \ge \alpha_*>0$ such that $\alpha_* \le
 \alpha_i(r,t)\leq \alpha^*$ for all $r\in I_B,$ $t>0,$ and $i=1,2.$
\end{itemize}

Then we have the following result.

\begin{theo}\label{main:theorem:neumann}
Let $(h1)$--$(h3)$ be satisfied, and let $u_1,u_2\in C^{2,1}(\overline{B}\times(0,\infty))\cap
C(\overline{B}\times[0,\infty))$ be functions such that $u=(u_1,u_2)$ solves
\eqref{model:competitive:neumann} and 
\begin{equation}
\label{eq:12}
\|u_i\|_{L^\infty(B\times(0,\infty))}<\infty \qquad \text{for $i=1,2$.}
\end{equation}
Suppose furthermore that assumption $(h0)$ of
Theorem~\ref{corollary:Lotka:Volterra} holds. Then there is
some $p\in \Sn$ such that every $(z_1,z_2)\in \omega(u)$ has
the property that $z_1$ is foliated Schwarz symmetric with respect to $p$ and $z_2$ is foliated Schwarz symmetric with respect to $-p.$
\end{theo}

As mentioned before, Theorem~\ref{corollary:Lotka:Volterra} is an
immediate consequence of Theorem~\ref{main:theorem:neumann}. As far as we know, there
is no previous result on the asymptotic symmetry of
competition-diffusion parabolic systems. For a related class of
Dirichlet problems for 
elliptic competing systems with a variational structure, Tavares and the second author proved
recently in \cite{weth:tavares} that the ground state solutions are
foliated Schwarz symmetric with respect to antipodal points. Note
that, in contrast, the elliptic counterpart of
(\ref{Lotka:Volterra:system}) has no variational structure which could
lead to symmetry information. 
More is known in the case of Dirichlet problems for {\em
  cooperative} systems. In particular, for a class of parabolic cooperative
systems, F\"{o}ldes and Pol\'{a}\v{c}ik 
\cite{polacik:systems} proved that, in the case where the underlying domain
is a ball, positive solutions are asymptotically radially symmetric and
radially decreasing. Moreover, for elliptic cooperative systems with
variational structure and some convexity properties of the data, Damascelli and Pacella \cite{pacella} proved foliated Schwarz symmetry of solutions having Morse index less or equal to the dimension of the domain. 

To prove Theorems \ref{main:theorem:scalar} and
\ref{main:theorem:neumann}, we follow the strategy of our previous
work \cite{saldana:weth} on a scalar Dirichlet problem, using a
rotating plane argument.  However, the proofs in \cite{saldana:weth}
rely strongly on parabolic maximum principles for small domains
due to Pol\'{a}\v{c}ik  \cite{polacik}, and these are only available under
Dirichlet boundary conditions. In the present paper, we replace this
tool by a Harnack-Hopf type estimate, 
Lemma~\ref{hopf:lemma:Neumann} below, which yields information up to the
nonsmooth part of the boundary of cylinders over half balls and half annuli. With the
help of this tool we show a stability property of reflection
inequalities with respect to small perturbations of a hyperplane, see
Lemma~\ref{perturbationlemma} below. 

The adjustment of the rotating plane method to systems gives rise to a further difficulty. When dealing with the
so-called semi-trivial limit profiles, that is, elements of
$\omega(u_1,u_2)$ of the form $(z,0)$ and $(0,z),$ the perturbation
argument within the rotating plane method cannot be carried out directly. To overcome this
obstacle, we apply a new normalization procedure and distinguish
different cases for the asymptotics of the normalized profile.  We
remark that a similar normalization argument can be made for the
Dirichlet problem version of system \eqref{model:competitive:neumann}.
In this case, the estimates given in \cite{huska:polacik:safonov} play
a decisive role, and the argument is somewhat more
technical.  To keep this paper short we do not include the Dirichlet case here.
We note that the occurrence and nature of semi-trivial limit profiles
have been studied extensively in recent years, see
e.g. \cite{cantrell,langa,dancer,lou,dancer:zhang}. 

It is natural to ask whether similar symmetry properties are available
for the cooperative version of problem \eqref{model:competitive:neumann}, i.e.,
\begin{equation}\label{model:cooperative:neumann}
\begin{aligned}
 (u_i)_t-\mu_i(|x|,t)\Delta u_i &=f_i(t,|x|,u_i)+\alpha_i(|x|,t) u_1u_2 && \text{in $B \times (0,\infty)$},\\
\partial_\nu u_i&=0&& \text{on $\partial B \times (0,\infty)$},\\
 u_i(x,0)&=u_{0,i}(x) \ge 0&& \text{for $x\in B,$ $i=1,2$.}
\end{aligned}
\end{equation}
Indeed, the proof of Theorem~\ref{main:theorem:neumann} can easily
be adjusted to deal with \eqref{model:cooperative:neumann}. More
precisely, we have the following result. 

\begin{theo}\label{main:theorem:neumann-cooperative}
Let $(h1)$--$(h3)$ be satisfied, and let $u_1,u_2\in C^{2,1}(\overline{B}\times(0,\infty))\cap
C(\overline{B}\times[0,\infty))$ be functions such that $u=(u_1,u_2)$ solves
\eqref{model:cooperative:neumann} and satisfies \eqref{eq:12}. Suppose
furthermore that 
$$
\leqno{\rm (h0)'} \qquad u_{0,i} \geq u_{0,i}
  \circ \sigma_e\ \ \  \text{ in }B(e) \ \ \ \ \left\{
  \begin{aligned}
&\text{for some $e \in \Sn$ with }\\
&\text{ $u_{0,i}\not\equiv u_{0,i}\circ \sigma_e$ for $i=1,2$.}    
  \end{aligned}
\right.   
$$
Then there is some $p\in \Sn$ such that every $(z_1,z_2)\in \omega(u)$ has
the property that $z_1,z_2$ are foliated Schwarz symmetric with respect to
$p$.
\end{theo}

The paper is organized as follows. In Section \ref{technical:lemmas}
we collect some preliminary tools which are rather easy consequences
of already established results. In Section~\ref{hopflemma} we 
derive a Harnack-Hopf type estimate for scalar equations in
a half cylinder under mixed boundary conditions, a
related version of the Hopf Lemma for cooperative systems and a perturbation
lemma for hyperplane reflection inequalities. In Section
\ref{results:scalar:equations} we complete the proof of
Theorem~\ref{main:theorem:scalar}, and in
Section~\ref{normalization:argument} we complete the (more difficult)
proof of Theorem~\ref{main:theorem:neumann}.
In Section~\ref{sec:other:problems} we first provide the proof of 
Theorem~\ref{main:theorem:neumann-cooperative} and then briefly
discuss further classes of competitive and cooperative systems
(see \eqref{cubic:system} and \eqref{general:cooperative:model} below).\\

\noindent \textbf{Acknowledgements:}
The work of the first author is supported by a joint grant
from CONACyT (Consejo Nacional de Ciencia y Tecnolog\'{\i}a - Mexico)
and DAAD (Deutscher Akademischer Austausch Dienst - Germany). The
authors would like to thank Nils Ackermann, Sven Jarohs, Filomena Pacella, Peter
Pol\'{a}\v{c}ik, and Hugo Tavares for helpful discussions related to the paper.
The authors also wish to thank the referee for his/her helpful comments.  

\section{Preliminaries}
\label{technical:lemmas}
First we fix some notation. Throughout the paper, we assume that $B$ is a ball or an annulus in
$\R^N$ centered at zero, 
and we fix $0 \le A_1 <A_2< \infty$ such that 
\begin{equation}\label{B:definition}
B:= \begin{cases}
\{x\in\mathbb R^N: A_1<|x|<A_2\}, & \text{ if } A_1>0,\\
 \{x\in\mathbb R^N: |x|<A_2\}, & \text{ if } A_1=0.
\end{cases}
\end{equation}
Note that $I_B=[A_1,A_2]$. For $\Omega\subset\mathbb R^N,$ we
let $\Omega^\circ$ denote the interior of $\Omega.$
For two sets $\Omega_1,\Omega_2 \subset \R^N$, we put
$\dist(\Omega_1,\Omega_2):= \inf \{|x-y|\::\: x \in \Omega_1,\,y \in
\Omega_2\}$. If $\Omega_1= \{x\}$ for some $x \in \R^N$, we simply
write $\dist(x,\Omega_2)$ in place of $\dist(\{x\},\Omega_2)$.

We will need equicontinuity properties of uniformly bounded global
solutions of (\ref{model:scalar}) and
(\ref{model:competitive:neumann}) and their gradients. These
properties are derived from standard uniform regularity estimates as collected in the following 
lemma.

\begin{lemma}\label{regularity} Let $\Omega \subset \mathbb R^{N}$ be
  a smooth bounded domain, $I\subset \R$ open, $\mu\in
  C^{1}(\Omega\times I),$ $g\in L^\infty(\Omega\times I),$ and let
  $v\in C^{2,1}(\overline{\Omega}\times I)\cap C(\overline{\Omega \times I})$ be a classical solution of
\begin{equation*}
\begin{aligned}
v_t-\mu(x,t)\Delta v&=g(x,t) &&\qquad \text{in $\Omega \times I$},\\
\partial_\nu v&=0 &&\qquad \text{on $\partial \Omega \times I$.}
\end{aligned}
\end{equation*}
Suppose moreover that 
\begin{equation*}
\begin{aligned}
&\mu_*:= \inf_{\Omega \times I} \mu(x,t)>0,\\
&K:= \|v\|_{L^\infty(\Omega\times I)}+\|\mu\|_{C^1(\Omega\times I)}+\|g\|_{L^\infty(\Omega\times I)}<\infty.
\end{aligned}
\end{equation*}
Let $h\in \{v,v_{x_j}:j=1,\ldots,N\}$ and ${\cal I}\subset I$
with $\operatorname{dist}({\cal I},\partial I)\geq 1$. Then there exist positive constants $C$ and $\gamma,$ depending only on $\Omega,$ $\mu_*,$ and $K,$ such that 
\begin{align}\label{equicontinuity:h}
\sup_{\genfrac{}{}{0pt}{}{\scriptstyle{x,\bar x\in
        \overline{\Omega},\, t,\bar t\in[t_0,t_0+1],}}{\scriptstyle{x \not= \bar x,\, t \not=
        \bar t,\, {t_0\in {\cal I}}}}}
 \frac{|h(x,t)-h(\bar x,\bar t)|}{|x-\bar x|^\gamma+|t-\bar
  t|^{\frac{\gamma}{2}}}< C.
\end{align}
\end{lemma}
\begin{proof}
Fix $t_0\in {\cal I}$ and set $Q:=\overline{\Omega}\times[t_0,t_0+1].$ Then, by \cite[Theorem 7.35, p.185]{lieberman} there is a constant $C>0,$ which depends only on $\Omega,$ $\mu_*,$ and $K$ such that
\begin{align*}
 \|D^2 v\|_{L^{N+3}(Q)}+\|v_t\|_{L^{N+3}(Q)}&\leq C(\|g\|_{L^{N+3}(Q)}+\|v\|_{L^{N+3}(Q)})\leq 2 C|\Omega|K.
\end{align*}
In particular, there is some constant $\tilde K>0$ independent of $t_0$ such that
\begin{align*}
 \|v\|_{W_{N+3}^{2,1}(Q)}\leq \tilde K.
\end{align*}
Next, fix $0 < \gamma < 1-\frac{N+2}{N+3}\in(0,1).$ By Sobolev embeddings (see, for example, \cite[embedding (1.2)]{quittner-souplet}) there exists  a constant $\tilde C>0$ which only depends on $\Omega$ such that
\begin{equation}
\label{regularity-new-equation}
 \|v\|_{C^{1+\gamma,(1+\gamma)/2}(Q)}\leq \tilde C \|v\|_{W_{N+3}^{2,1}(Q)}\leq \tilde C \tilde K,
\end{equation}
where 
$$
\|u\|_{C^{1+\gamma,(1+\gamma)/2}(Q)}:= \|u\|_{L^\infty(Q)}+ |u|_{\gamma;Q}+ \|\nabla u\|_{L^\infty(Q)}+ |\nabla u|_{\gamma;Q}
$$
and 
$$
|v|_{\gamma;Q}:= \sup\bigg\{ \frac{|v(x,t)-v(y,s)|}{|x-y|^\gamma+|t-s|^\frac{\gamma}{2}}\::\: (x,t),(y,s)\in \overline{Q},\ (x,t)\neq (y,s)\bigg\} 
$$
for functions $v: Q \to \R$ resp. $v: Q \to \R^N$. Since the constant $\tilde C \tilde K$ in (\ref{regularity-new-equation}) does not depend on the choice
of $t_0,$ we obtain \eqref{equicontinuity:h}.
\end{proof}

\begin{remark}\label{equicontinuity}
If $u=(u_1,u_2)$ is a nonnegative solution of \eqref{model:competitive:neumann}
 with $u_1,u_2\in C^{2,1}(\overline{B}\times(0,\infty))\cap
C(\overline{B}\times[0,\infty))$, then $u_i$ satisfies
\begin{equation*}
 (u_i)_t-\mu_i(|x|,t)\Delta u_i = g_i(x,t),\qquad x\in B,\ t>0,
\end{equation*}
with $g_i: B \times (0,\infty) \to \R$ given by 
$$
g_i(x,t)=f_i(t,|x|,u_i(x,t))-\alpha_i(|x|,t)u_1(x,t)u_2(x,t)\qquad \text{for $i=1,2$.}
$$
If, moreover, (h1)-(h3) and (\ref{eq:12}) are satisfied, then we have 
$\|u_i\|_{L^\infty(B \times (0,\infty))}<\infty$ and 
$\|g_i\|_{L^\infty(B \times (0,\infty))}< \infty$ for $i=1,2$. Since also the diffusion coefficients $\mu_i$ satisfy the
assumptions of Lemma~\ref{regularity}, we conclude that $u_i$ and $\partial_j u_i$ for all
$j=1,\dots,N,$ $i=1,2,$ satisfy \eqref{equicontinuity:h} with $\cI=(1,\infty).$ 
As a consequence of \eqref{equicontinuity:h}, the
semiorbits $\{u_i(\cdot,t):t\geq 1\}$ are precompact sets for
$i=1,2$. Hence, by a standard
compactness argument, the omega limit sets of $u=(u_1,u_2)$ and its
components are related as follows:
\begin{equation}
  \label{eq:3}
\omega(u_i)= \{z_i \::\: z=(z_1,z_2) \in \omega(u)\} \qquad \text{for $i=1,2$.}   \end{equation}
\end{remark}
\medskip

Next we define extensions of solutions to second order Neumann
problems on $B$ to a larger domain via inversion at
the boundary. Recalling (\ref{B:definition}), we define 
\begin{equation}
\label{eq:14}
\widetilde B:= \begin{cases}
\{x\in\mathbb R^N: \frac{A_1^2}{A_2}<|x|<\frac{A_2^2}{A_1}\}, & \text{ if } A_1>0,\\
 \{x\in\mathbb R^N: |x|<2A_2\}, & \text{ if } A_1=0,
\end{cases}
\end{equation}
and for $x \in \widetilde B \setminus B$ we put 
$$
\hat x := 
\begin{cases}
\frac{A_2^2}{|x|^2}x, & \text{ if $|x| \ge A_2$,}\\
\frac{A_1^2}{|x|^2}x, & \text{ if $|x| \le A_1$.}
\end{cases}
$$

\begin{lemma} \label{extension:lemma} 
Let $I \subset \R$ be an open interval, $\mu,g: B\times I\to \R$ be
given functions and let $u\in C^{2,1}(\overline{B}\times I)\cap
C(\overline{B \times I})$ be a solution of 
$$
\left \{
  \begin{aligned}
u_t-\mu(x,t)\Delta u &= g(x,t)&& \qquad \text{in $B \times I$,}\\
 \partial_\nu u &=0 &&\qquad \text{on $\partial B \times I.$}
  \end{aligned}
\right.
$$
Then the function 
\begin{equation}
  \label{eq:16}
\tilde u: \widetilde B \to \R,\qquad 
\tilde u(x,t):= 
\begin{cases}
 u(x,t), & x\in \overline{B},\ t\in I,\\\vspace{.1cm}
 u(\hat x,t), & x\in \widetilde B \setminus \overline {B},\ t\in I,
\end{cases}
\end{equation}
satisfies that $\tilde u\in W^{2,1}_{p,loc}(\widetilde B\times I)\cap
C^{1}(\widetilde B\times I)$ for any $p\geq 1$, and it is a strong solution of the equation
\begin{align}
\label{extended:neumann:system}
{\tilde u}_t-\tilde \mu(x,t)\Delta \tilde u-\tilde b(x,t) \partial_r \tilde u &= \tilde g(x,t) \qquad \text{ in } \widetilde B\times I.
 \end{align}
Here $\partial_r = \frac{1}{|x|} \sum \limits_{j=1}^N x_j \partial_j$ is the
radial derivative and 
\begin{align*}
\tilde \mu(x,t)&:=\begin{cases}
\mu(x,t), & \qquad x\in B,\ t\in I, \\\vspace{.1cm}
\frac{|x|^2}{|\hat x|^2}\mu(\hat x,t\big), &\qquad x\in\widetilde B
\setminus B,\ t\in I,
\end{cases}\\
\tilde b(x,t)&:=\begin{cases}
0, & x\in B,\ t\in I\\\vspace{.1cm}
 \frac{(4-2N)|x|}{|\hat x|^2} \mu(\hat x,t), & x\in\widetilde B
  \setminus B,\ t\in I,
\end{cases}\\
\tilde g(x,t)&:=\begin{cases}
g(x,t), & \qquad \qquad x\in B,\ t\in I,\\\vspace{.1cm}
g(\hat x,t), & \qquad \qquad x\in\widetilde B \setminus B,\ t\in I.
\end{cases}
\end{align*}
\end{lemma}

\begin{proof} As a consequence of the Neumann boundary conditions we have 
$\tilde u\in C^{1}(\widetilde{B} \times I)$. Fix $p\geq 1.$ By assumption, $\|u\|_{W_{p}^{2,1}(B\times J)}<\infty$ for any subinterval $J\subset\subset I.$ Since the map $x \mapsto \hat x$ has uniformly
bounded first and second derivatives in $\tilde B \setminus B$, it follows that $\|\tilde u\|_{W_{p}^{2,1}(\tilde B\times
  J)}<\infty.$ Finally, it is easy to check by direct calculation that
\eqref{extended:neumann:system} holds for $x\in \widetilde B
\setminus \partial B$ and $t\in I$. Combining these facts,
we find that $\tilde u$ is a strong solution of (\ref{extended:neumann:system}). 
\end{proof}

\begin{remark}\label{extension:remark}
A similar extension property is valid in half balls and half annuli under mixed
boundary conditions. More precisely, let $B_+:= \{x \in \overline
B\::\:x_N>0\}$, $I \subset \R$ be an open interval, let $\mu,g:B_+\times I\to R$ given functions and let $u\in C^{2,1}(\overline{B_+}\times I)\cap
C(\overline{B_+\times I})$ be a solution of 
$$
u_t-\mu(x,t)\Delta u = g(x,t) \qquad \text{in $B_+^\circ \times I$,}\\
$$
satisfying $u=0$ on $\Sigma_1\times I$ and 
$\partial_\nu u =0$ on $\Sigma_2\times I,$ where 
\begin{equation}
  \label{eq:19}
\Sigma_1 := \{x\in\partial B_+ : x_N=0\}, \qquad \Sigma_2 :=
\{x\in\partial B_+ : x_N>0\}.
\end{equation}
Let $\widetilde {B_+}:= \{x \in \widetilde B\::\: x_N>0\}$ and define
$\tilde u: \widetilde {B_+} \to \R$ by (\ref{eq:16}) for $x \in
\widetilde {B_+}$. Then $\tilde u\in W^{2,1}_{p,loc}(\widetilde {B_+}\times I)\cap
C^{1}(\widetilde {B_+} \times \overline{I})$ for any $p>N+2$ and it is a strong solution of
(\ref{extended:neumann:system}) in $\widetilde {B_+}$ with
coefficients defined analogously as in Lemma~\ref{extension:lemma}.
\end{remark}

The final preliminary tool we need is a geometric characterization of a set of foliated Schwarz
symmetric functions. We first recall the following result from \cite[Proposition 3.2]{saldana:weth}.

\begin{prop}
\label{sec:char-foli-schw-1}
 Let $\cU$ be a set of continuous functions defined on a radial domain $B\subset
 \mathbb R^N,$ $N\geq 2,$ and suppose that there exists 
 \begin{equation}
   \label{eq:5}
\tilde e\in \cM_\cU:=\{e\in \Sn \mid z(x) \ge z(\sigma_e(x)) \text{ for all }x\in B(e) \text{ and } z \in \cU\}.
 \end{equation}
If for all two dimensional subspaces $P\subseteq
 \mathbb R^N $ containing $\tilde e$ there are two different points
 $p_1, \ p_2$ in the same connected component of $\cM_\cU\cap P$ such that
 $z \equiv z \circ \sigma_{p_1}$ and $z \equiv z \circ
 \sigma_{p_2}$ for every $z \in \cU$, then there is $p
 \in \Sn$ such that every $z \in \cU$ is foliated Schwarz symmetric with
 respect to $p$.
\end{prop}

Instead of applying this Proposition directly, we will rather use the
following corollary. 

\begin{coro}
\label{sec:symm-char}
 Let $\cU$ be a set of continuous functions defined on a radial domain $B\subset
 \mathbb R^N,$ $N\geq 2,$ and suppose that the set $\cM_\cU$ defined in
 (\ref{eq:5}) contains a nonempty subset $\cN$ with the following
 properties: 
 \begin{itemize}
 \item[(i)] $\cN$ is relatively open
 in $\Sn$;
\item[(ii)] For every $e \in \partial \cN$ and $z \in
 \cU$ we have $z \le z \circ \sigma_e$ in $B(e)$. Here $\partial
 \cN$ denotes the relative boundary of $\cN$ in $\Sn$. 
 \end{itemize}
Then there is $p \in \Sn$ such that every $z \in \cU$ is foliated Schwarz symmetric with
 respect to $p$.
\end{coro}

\begin{proof}
By assumption, there exists $\tilde e \in \cN \subset \cM_\cU$. Let $P\subseteq
 \mathbb R^N$ be a twodimensional subspace containing $\tilde
 e$, and let $L$ denote the connected component of $\overline {\cN \cap
   P}$ containing $\tilde e$. Since $\cM_\cU$ is
closed, $L$ is a subset of the connected component of $\cM_\cU \cap P$
containing $\tilde e$. By Proposition~\ref{sec:char-foli-schw-1}, it suffices to show that 
there are different points $p_1,p_2 \in L$ such that   
$z \equiv z \circ \sigma_{p_1}$ and $z \equiv z \circ
 \sigma_{p_2}$ for every $z \in \cU$.
 
We distinguish two cases. If $L= \Sn \cap P$, then we have $z \equiv z
\circ \sigma_p$ in $B$ for every $p \in L$, $z \in \cU$ by the definition of $\cM_\cU$
and since $L \subset \cM_\cU$.

 If $L \not = \Sn \cap P$, then there exists two different points
 $p_1,p_2$ in the relative boundary of $L$ in $\Sn \cap
 P$. Since $\cN$ is relatively open in $\Sn$, these points are contained
 in $\partial \cN \subset \cM_\cU$, and by assumption and the definition
 of $\cM_\cU$ we have $z \equiv z \circ \sigma_{p_1}$ and $z \equiv z \circ
 \sigma_{p_2}$ in $B$ for every $z \in \cU$, as required.
\end{proof}

\section{A Harnack-Hopf type lemma and related estimates}\label{hopflemma}

The first result of this section is an estimate related to a linear parabolic boundary
value problem on a (parabolic) half
cylinder. The estimate can be seen as an extension of
both the Harnack inequality and the Hopf lemma since it also gives information on a \textquotedblleft tangential
\textquotedblright derivative at corner points. A somewhat related
(but significantly weaker) result for supersolutions of the Laplace
equation was given in \cite[Lemma A.1]{girao:weth}. 

\begin{lemma}\label{hopf:lemma:Neumann}
Let $a,b \in \R$, $a<b$, $I:=(a,b),$ $B_+:=\{x\in \overline B\::\: x_N>0\}.$ Suppose that $v\in C^{2,1}(\overline{B_+}\times I)\cap C(\overline{B_+\times I})$ satisfies 
\begin{equation*}
\begin{aligned}
v_t-\mu\Delta v - c v&\geq 0 &&\qquad \text{in $B^\circ_+ \times I$,}\\
\frac{\partial v}{\partial \nu}&= 0 &&\qquad \text{on $\Sigma_2 \times I$,}\\
\hspace{2.7cm} v &= 0 && \qquad \text{on $\Sigma_1 \times I$,}\\
\hspace{2.7cm} v(x,a) &\geq 0&&\qquad \text{for $x\in B_+$,}
\end{aligned}
\end{equation*}
where the sets $\Sigma_i$ are given in (\ref{eq:19}) and the coefficients satisfy 
\begin{equation*}
\frac{1}{M} \leq \mu(x,t)\leq M \quad \text{and}\quad 
 |c(x,t)|\leq M \qquad \text{ for } (x,t)\in B_+\times I
\end{equation*}
with some positive constant $M>0.$ Then $v\geq 0$ in $B_+\times
(a,b)$. Moreover, if $v(\cdot,a) \not\equiv 0$ in $B_+$, then 
\begin{equation}\label{hopf:conclusion:general:notation}
v>0 \text{ in } B_+\times I\qquad  \text{and}\qquad \frac{\partial v}{\partial e_N}>0 \text{ on } \Sigma_1\times I.
\end{equation}
Furthermore, for every $\delta_1>0,$ $\delta_2 \in (0,\frac{b-a}{4}]$, there exist $\kappa>0$ and $p>0$ depending only on
$\delta_1,$ $\delta_2$, $B$ and $M$ such that 
\begin{equation}
  \label{eq:1}
v(x,t) \ge x_N \,\kappa \bigg(\int_{Q(\delta_1,\delta_2)} v^p \ dxdt\bigg)^{\frac{1}{p}}\ \ \  \text{for every
  $x \in B_+,\: t \in [a+3\delta_2,a+4\delta_2],$}   
\end{equation}
with $Q(\delta_1,\delta_2):= \{ (x,t) \::\: 
x \in B_+,\: x_N \ge \delta_1, \: a+
  \delta_2 \le t \le a+2\delta_2\}.$
\end{lemma}

\begin{proof} 
We begin by showing that $v\geq 0$ in $B_+\times I$. The argument is standard. Let $\varepsilon>0$ and define $\varphi(x,t):=e^{-2M t}v(x,t)+\varepsilon$ for $x\in \overline{B_+}$ and $t\in \overline{I}.$ Then 
\begin{align*}
\varphi_t-\mu\Delta \varphi-(c-2M)\varphi&\ge \varepsilon(2M-c )\geq \eps M, &&x\in B^\circ_+, t\in I,\\
 \frac{\partial \varphi}{\partial \nu}(x,t) &= 0, &&x\in\Sigma_2, t\in I,\\
 \varphi(x,t) &= \varepsilon, &&x\in \Sigma_1, t\in I,\\
 \varphi(x,a) &\geq \varepsilon, &&x\in B_+.
\end{align*}
Suppose by contradiction that $\bar t := \sup \{t \in [a,b)\::\: \text{$\varphi> 0$ in $B_+\times [a,t)$}\} \: <\: b.$ By continuity, we have $\bar t>a$, $\varphi(\cdot,\bar t)\geq 0$ in $B_+$ and $\varphi(\bar
x,\bar t)=0$ for some $\bar x\in B_+$. As a consequence of the Neumann boundary conditions on $\Sigma_2 \times I$  and the boundary point lemma (see for example \cite[Lemma 2.8]{lieberman}), we find that $\bar x\in {B_+}^\circ$ . But then 
\begin{align*}
0\geq \varphi_t(\bar x,\bar t)-\mu\Delta \varphi(\bar x,\bar t)-(c-2M)\varphi(\bar x,\bar t) \ge \eps M > 0,
\end{align*}
a contradiction. Therefore $\varphi>0 $ in $B_+\times I$. Since $\eps>0$ was chosen arbitrarily, we conclude that $v\geq 0$ in $B_+\times I$.  Then the first claim in
\eqref{hopf:conclusion:general:notation} follows by the strong maximum
principle and the boundary point lemma (see e.g. \cite[Theorem
2.7 and Lemma 2.8]{lieberman}).

Next we note that the second claim in
\eqref{hopf:conclusion:general:notation} is a consequence of the first
claim and the inequality (\ref{eq:1}) (for suitably chosen
$\delta_1,\delta_2$). It thus remains to prove (\ref{eq:1}). Let
$\delta_1>0$, $\delta_2 \in (0,\frac{b-a}{4}]$ and consider
$\widetilde B, \widetilde {B_+}$ as defined in (\ref{eq:14}) and Remark~\ref{extension:remark}. Without loss, we
may assume that 
\begin{equation}
  \label{eq:2}
\delta_1 < \min \Bigl \{\frac{\delta_2}{2}, \frac{\dist(B,\partial \widetilde
B)}{3}\Bigr \}. 
\end{equation}
By Remark \ref{extension:remark}, there exists an extension $\tilde
v\in W_{N+1,loc}^{2,1}(\widetilde {B_+} \times
I)$ of $v$ which satisfies $\cL(t,x)\tilde v\geq 0$ in
$\widetilde {B_+} \times I$ in the strong sense. Here the linear differential
operator $\cL$ is given by 
 \begin{align*}
 \cL(t,x) w:= w_t-\tilde \mu(x,t)\Delta w- \tilde b(x,t)\partial_r w-\tilde c(x,t)w
 \end{align*}
 with $\tilde \mu,\: \tilde b$ given as in Lemma \ref{extension:lemma}
 and 
\begin{equation*}
\tilde c(x,t):= \begin{cases}
c(x,t), &\quad x\in B_+,\ t\in I,\\\vspace{.1cm}
c(\hat x,t), &\quad x\in\widetilde B_+ \setminus B_+,\ t\in I.
\end{cases}
\end{equation*}  

Moreover, there is a positive constant $\beta_0$ which only depends on $B$ and $M$ such that $\tilde \mu,$ $\tilde b,$ and $\tilde c$ are uniformly bounded by $\beta_0,$ and $\tilde\mu$ is
 bounded below by $\beta_0^{-1}.$ Next, we define the compact sets 
\begin{align*}
K_{\delta_1}&:= \{x \in B_+\::\: x_N \ge \frac{\delta_1}{2}\}\quad \text{and}\\
\tilde K_{\delta_1}&:= \{x \in \widetilde {B_+} \::\: x_N\geq \frac{\delta_1}{2},\:\dist(x,\partial \widetilde B) \ge \delta_1 \}.
\end{align*}
By the parabolic Harnack inequality given in \cite[Lemma
3.5]{polacik}, there exist $\kappa_1>0$ and $p>0$, depending only
on $\delta_1, \delta_2$, $B$ and $M$, such that
\begin{align}\label{eq1:hopf-0}
\inf_{\genfrac{}{}{0pt}{}{\scriptstyle{x\in \tilde
       K_{\delta_1}}}{\scriptstyle{t\in[a+\frac{5}{2}\delta_2,a+4\delta_2]}}}\tilde
 v(x,t) \nonumber
&\geq \kappa_1 \bigg(\int_{\tilde K_{\delta_1} \times
  [a+\delta_2,a+2\delta_2]} {\tilde v}^p\ dxdt\bigg)^{\frac{1}{p}}\\
  &\geq \kappa_1 \bigg(\int_{Q(\delta_1,\delta_2)}  v^p\ dxdt\bigg)^{\frac{1}{p}}.
\end{align}
Here we used in the last step that 
$$
Q(\delta_1,\delta_2) \;\subset \; K_{\delta_1} \times
  [a+\delta_2,a+2\delta_2] \;\subset \; \tilde K_{\delta_1} \times
  [a+\delta_2,a+2\delta_2].
$$
Next, we define 
\begin{align*}\begin{split}
D&:=\{(x,t)\::\:t<0,\: x_N <
\frac{\delta_1}{2},\: |x-\delta_1 e_N|^2 + t^2 < \delta_1^2 \};\\
\Gamma_1&:= \{(x,t)\::\: t \le 0,\: x_N <
\frac{\delta_1}{2},\: |x-\delta_1 e_N|^2 + t^2 = \delta_1^2  \};\\
\Gamma_2&:= \{(x,t)\::\:t \le 0,\: |x-\delta_1 e_N|^2 + t^2 \le \delta_1^2, \: x_N =
\frac{\delta_1}{2} \}.
\end{split}
\end{align*}
Note that $\Gamma_1 \cup \Gamma_2$ equals $\partial_P D$, the
parabolic boundary of $D$. Let $x_0\in \Sigma_1$ and $t_0 \in
[a+3\delta_2,a+4 \delta_2]$. By construction and (\ref{eq:2}), we then have 
$$
\{(x_0+x,t_0+t)\::\: (x,t) \in D\} \subset \widetilde
{{B_+}}^\circ\times [a+\frac{5}{2}\delta_2,a + 4 \delta_2]
$$
and 
\begin{equation}
  \label{eq:17}
\{(x_0+x,t_0+t)\::\: (x,t) \in \Gamma_2\} \subset \tilde
K_{\delta_1} \times [a+\frac{5}{2}\delta_2,a + 4 \delta_2].
\end{equation}
Next we fix $k>0$ such that 
\begin{align*}
 k\geq \frac{2\beta_0[{\delta_1}+\beta_0N(1+{\delta_1})]}{{\delta_1}^2}.
\end{align*}
Moreover, we define the function
\begin{equation*}
z:\overline D\to\mathbb R,\qquad 
 z(x,t):= \Bigl(e^{-k(|x-{\delta_1}
   e_N|^2+t^2)}-e^{-k{\delta_1}^2}\Bigr)e^{-\beta_0 t} 
\end{equation*}
Let also 
$$
\eps:= \frac{\min \limits_{(x,t) \in \Gamma_2} \tilde v(x_0+x,t_0+t)}{\max \limits_{(x,t) \in \Gamma_2} z(x,t)} >0
$$
and consider  
\begin{align*}
w: \overline D\to\mathbb R,\qquad
 w(x,t):=\tilde v(x_0+x,t_0+t)-\varepsilon z(x,t) 
\end{align*}
Then $w \geq 0$ on $\Gamma_2$ and also $w\geq 0$ on $\Gamma_1$, since
$z\equiv 0$ on $\Gamma_1$. 
Moreover, for $(x,t)\in D$ we have 
\begin{align*}
&\cL(t_0+t,x_0+x)z(x,t)\\
= &[-\beta_0 - \tilde c(t+t_0,x_0+x)]z(x,t)\\
& + 2k\, e^{-k(|x-{\delta_1} e_N
  |^2+t^2)-\beta_0 t}\Bigl[\tilde \mu(t_0+t,x_0+x)(N-2k|x-{\delta_1} e_N|^2)   
\\&-t -\tilde b(t_0+t, x_0+x)\frac{x_0+x}{|x_0+x|}\cdot (x-{\delta_1} e_N)\Bigr]\\
\leq& 2k\, e^{-k(|x-{\delta_1} e_N|^2+t^2)-\beta_0t}\Bigl[{\delta_1}-
\frac{2k}{\beta_0} \Bigl(\frac{{\delta_1}}{2}\Bigr)^2 + \beta_0 N(1 + {\delta_1})\Bigr]\leq 0,
\end{align*}
by the definition of $k.$ Therefore we have 
$$
\cL(t_0+t,x_0+x)w(x,t) \geq 0 \;\: \text{for $(x,t)\in D$}\quad
\text{and}\quad w\geq 0\;\: \text{on $\partial_P D = \Gamma_1\cup \Gamma_2$.}
$$
By the maximum principle for strong solutions,  we conclude that
$w\geq 0$ in $\overline D$ and thus in particular 
$$
\tilde v(x_0+s e_N,t_0) \ge \eps z(s e_N,0) \qquad \text{for $s \in (0,\frac{{\delta_1}}{2})$.}
$$
Since moreover
$$
z(s e_N,0)= e^{-k(s-{\delta_1})^2}-e^{-k{\delta_1}^2} \ge s\, \eps_1\, \max \limits_{(x,t) \in \Gamma_2} z(x,t)  \qquad \text{for $s \in (0,\frac{{\delta_1}}{2})$}
$$
with a constant $\eps_1 \in (0,\frac{1}{\operatorname{diam}(B)})$ depending
only on the function $z$ and on $B$, it follows that 
$$
\tilde v(x_0+s e_N,t_0) \ge s \eps_1\, \eps \max \limits_{(x,t) \in \Gamma_2} z(x,t)  = \eps_1 s  \min \limits_{(x,t) \in
  \Gamma_2} \tilde v(x_0+x,t_0+t)  \quad \text{for $s \in (0,\frac{{\delta_1}}{2})$.}
$$
By (\ref{eq:17}) and since $x_0 \in \Sigma_1$, $t_0 \in
[a+3\delta_2,a+4 \delta_2]$ were chosen arbitrarily, we conclude that 
$$
v(x,t) \:\ge\: \eps_1 x_N \!\!\!\!\inf_{\genfrac{}{}{0pt}{}{\scriptstyle{y \in \tilde
       K_{\delta_1}}}{\scriptstyle{\tau  \in [a+\frac{5}{2}\delta_2, a+4\delta_2]}}}\!\!\!\!\tilde
 v(y,\tau) \qquad 
\left\{
  \begin{aligned}
  &\text{for $x \in B_+$ with $x_N < \frac{{\delta_1}}{2}$}\\
  &\text{and $t \in [a+3\delta_2,a+4 \delta_2].$}    
  \end{aligned}
\right.
$$
 By definition of $K_{\delta_1}$ and since 
$0 \le \eps_1 x_N \le 1$ for $x \in B_+$, the latter
estimate holds also without the restriction $x_N <
\frac{{\delta_1}}{2}$. Combining this fact with (\ref{eq1:hopf-0}), we obtain that 
\begin{align*}
v(x,t) \ge \kappa_1 \eps_1 x_N \bigg(\int_{Q(\delta_1,\delta_2)} v^p \
dxdt\bigg)^{\frac{1}{p}}\ \ \ \ \text{for $x \in B_+$ and $t \in [a+3\delta_2,a+4 \delta_2],$}
\end{align*}
 so that (\ref{eq:1}) holds with $\kappa:= \kappa_1 \eps_1$.
\end{proof}

Next, we prove a related but weaker Hopf Lemma for a class of
cooperative systems under mixed boundary conditions. The argument is
essentially the same as in the scalar case, but we include it for
completeness since we could not find the result in this form in the literature. We use the notation of Lemma \ref{hopf:lemma:Neumann}.

\begin{lemma}\label{hopf:lemma:Neumann:systems}
Let $a,b \in \R$, $a<b$, $I:=(a,b),$ $J:=\{1,2,\ldots,n\}$ for some $n\in\mathbb N,$ and $w=(w_1, w_2,\ldots, w_n)$ with $w_i\in C^{2,1}(\overline{B_+}\times I)\cap C(\overline{B\times I})$ be a classical solution of 
\begin{equation*}
\begin{aligned}
(w_i)_t-\mu_i\Delta w_i&=\sum_{j\in J}c_{ij}w_j&&\qquad \text{in
  $B_+^\circ \times I$,}\\
\frac{\partial w_i}{\partial \nu}&= 0 &&\qquad \text{on $\Sigma_2 \times I$,}\\
\hspace{2.7cm} w_i &= 0 &&\qquad \text{on $\Sigma_1 \times I$,}\\
\hspace{2.7cm} w_i(x,\tau) &\geq 0&&\qquad \text{for $x\in B_+$,}
\end{aligned}
\end{equation*}
for $i\in J$ with coefficient functions $\mu_i, c_{ij}\in
L^\infty(B_+\times I)$. Suppose moreover that $\inf\limits_{B_+ \times
  I}\mu_i>0$ and $\inf\limits_{B_+ \times
  I}c_{ij}\geq 0$ for $i,j \in
J$, $i\neq j$. Then 
\begin{equation}
  \label{eq:10}
w_i\geq 0 \qquad \text{in $B_+\times I$ for $i\in J$.}  
\end{equation}
Moreover, if $w_i(x,\tau) \not\equiv 0$ for some $i\in J,$ then 
\begin{equation}\label{hopf:conclusion:systems}
w_i>0 \text{ in } B_+\times I\ \ \ \  \text{ and }\ \ \ \  \frac{\partial w_i}{\partial e_N}>0 \text{ on } \Sigma_1\times I.
\end{equation}
\end{lemma}

 \begin{proof} To prove (\ref{eq:10}), we fix \mbox{$\lambda > \max \limits_{i\in J} \sum
     \limits_{j\in
     J}\|c_{ij}\|_{L^\infty(B_+\times I)}$} and let
 $\varepsilon>0$. We define $v_i(x,t):=e^{-\lambda t}w_i(x,t)+\varepsilon$ for $x\in
   \overline{B_+}, t\in \overline{I}$ and $i \in J$. 
Then  
\begin{align*}
(v_i)_t-\mu_i\Delta v_i-(c_{ii}-\lambda)v_i&>\sum_{j\in
  J\backslash\{i\}}c_{ij}v_j &&\qquad \text{in ${B_+}^\circ \times I$,}\\
 \frac{\partial v_i}{\partial \nu} & \equiv  0 &&\qquad \text{on
   $\Sigma_2 \times I,\quad$ and}\\
 v_i &\ge   \varepsilon>0&& \qquad \text{on $\Sigma_1 \times I \;\cup\;
   B_+ \times \{a\}$.}
\end{align*}
As in the proof of Lemma \ref{hopf:lemma:Neumann}, we show that $v_i> 0$ in
$B_+\times I$ for all $i\in J.$  Suppose by contradiction that 
$$
\bar t := \sup \{t \in [a,b)\::\: \text{$v_i> 0$ in $B_+\times [a,t)$
  for all $i\in J$}\}\;<\;b.
$$
By continuity, we have $\bar t>a$, $v_i(\cdot,\bar t)\geq 0$ in $B_+$
for all $i \in J$ and $v_j(\bar
x,\bar t)=0$ for some $\bar x\in B_+$ and some $j \in J$. The Neumann boundary conditions on $\Sigma_2 \times I$  and the boundary point lemma (see for example \cite[Lemma 2.8]{lieberman}) then imply that $\bar x\in {B_+}^\circ$. But then 
\begin{align*}
0\geq (v_j)_t(\bar x,\bar t)-\mu_i\Delta v_i(\bar x,\bar t)-(c_{ii}-\lambda)v_i(\bar x,\bar t)>\sum_{j\in J\backslash\{i\}}c_{ij}v_j(\bar x,\bar t)\geq 0,
\end{align*}
a contradiction. Therefore $v_i>0 $ in $B_+\times I$ for all $i \in
J$. Since $\eps>0$ was chosen arbitrarily, we conclude that
(\ref{eq:10}) holds. Consequently, the non-negativity of $c_{ij}$ for
$i \not=j$ implies that 
\begin{align*}
(w_i)_t-\mu_i(x,t)\Delta w_i- c_{ii}(x,t)w_i&=\sum_{j\in
  J\backslash\{i\}} c_{ij}w_j \geq 0 \qquad \text{in $B_+\times I$ for $i\in J.$} 
\end{align*}
Hence (\ref{hopf:conclusion:systems}) follows from Lemma~\ref{hopf:lemma:Neumann}.
\end{proof}

For the last lemma of this section, we need to fix additional
 notation. For $e\in \Sn,$ let $\sigma_e:\overline{B}\to \overline{B}$ and $B(e) \subset B$ be defined as in the
introduction. We also put
\begin{equation}
  \label{eq:11}
\Sigma_1(e) := \{x\in\partial B(e) : x\cdot e=0\} \;\:
\text{and}\;\: \Sigma_2(e) := \{x\in\partial B(e) : x\cdot e>0\}.
\end{equation}
For a subset $I\subset \R$ and a function $v:\overline{B}\times I \to
\R$, we define 
$$
v^e: \overline{B} \times I \to \R, \qquad v^e(x,t):= v(x,t)-v(\sigma_e(x),t).
$$
To implement the rotating plane technique for the boundary value problems
considered in our main results, we need to analyze under which
conditions positivity of $v^e(t,\cdot)$ in $B(e)$ at some time $t
\in I$ induces positivity of $v^{e'}(t',\cdot)$ in $B(e')$ 
for a slightly perturbed direction $e'$ at a later time $t'>t$. The following perturbation lemma is sufficient for our purposes.

\begin{lemma}\label{perturbationlemma}
Let $I= (0,1)$, let $v \in
C^{2,1}(\overline{B\times I})$, and consider a function \mbox{$\chi :[0,\sqrt{1+\operatorname{diam}(B)^2}\,] \to
  [0,\infty)$} such that  
$$
\leqno{(E\chi)} \qquad \left\{
  \begin{aligned}
&\text{$\lim \limits_{\vartheta \to 0} \chi(\vartheta)=0$  and}\\ 
&|v(x,t)-v(y,s)|+|\nabla v(x,t)-\nabla v(y,s)|\leq \chi(|(x,t)-(y,s)|)\\
&\text{for all $(x,t),(y,s)\in \overline {B}\times I.$}
  \end{aligned}
\right.
$$
Moreover, let $d,k,M>0$ be given constants. Then there
exists $\rho>0$, depending only on $B$, $d$, $k$, $M,$ and the function $\chi,$ with the
following property:  If $e \in \Sn$ is such that 
\begin{enumerate}
\item[(i)] the function $v^e$ satisfies 
\begin{equation*}
v^e_t-\mu(x,t)\Delta v^e - c(x,t) v^e \geq 0 \qquad \text{in $B(e)
  \times I$}
\end{equation*}
with some coefficient functions $\mu, c$ satisfying
\begin{equation*}
\frac{1}{M} \leq \mu(x,t)\leq M \quad \text{and}\quad 
 |c(x,t)|\leq M \qquad \text{ for } (x,t)\in B(e) \times I,
\end{equation*}
and 
\begin{equation*}
\text{$\frac{\partial v^e}{\partial \nu}= 0$ on $\Sigma_2(e) \times I$, $\ \ \; v^e = 0$ on $\Sigma_1(e) \times I,$  
$\ \ \; v^e \geq 0$ on  $B(e) \times \{0\}$,
}
\end{equation*} 
\item[(ii)] $\quad \sup \{v^e(x,\frac{1}{4})\::\: x \in B(e),\: x \cdot e \ge d \} \ge k,$
\end{enumerate}
then  
\begin{equation*}
v^{e'}(\cdot,1)>0 \quad \text{in $B(e')\quad$ for all $e'\in \Sn$  with $|e-e'|<\rho$.}
\end{equation*}
\end{lemma}

\begin{remark}
\label{sec:harnack-hopf-type}
The result obviously remains true if $v^e$ is replaced by $-v^e$, and
we will use this fact later on.
\end{remark}

\begin{proof} 
Let $e\in \Sn$ be such that $(i)$ and $(ii)$ are satisfied, and let
$\kappa>0$ and $p>0$ be the constants  given by Lemma~\ref{hopf:lemma:Neumann} applied to $a=0$, $b=1$, $\delta_1=d$ and $\delta_2=\frac{1}{4}.$  We first note that condition $(E\chi)$ and hypothesis $(ii)$ imply that there exists $C_1>0$,
depending only on $B,$ $d$, $k$, $M,$ and $\chi,$ such that 
$$
\kappa\bigg(\int_{Q^e} (v^e)^p\,dx\,dt\bigg)^{\frac{1}{p}} \ge C_1,
$$
where $Q^e:= \{(x,t)\::\:  x \in B(e),\: x \cdot e \ge d,\: \frac{1}{4} < t
< \frac{1}{2} \}$. Then, by Lemma~\ref{hopf:lemma:Neumann}, it follows that
\begin{align*}
|\nabla v^e(x,1)|= \nabla v^e(x,1)  \cdot e \geq C_1\qquad \text{ for all } x\in
\Sigma_1(e).
\end{align*}
By condition $(E\chi)$, there is some
$\rho_0>0$, depending only on $B,$ $d$, $k$, $M,$ and $\chi,$ such that 
\begin{equation}\label{eq2:hopf-new1}
|\nabla v^{e'}(x,1)|= \nabla v^{e'}(x,1)  \cdot e' \geq \frac{3}{4}C_1 \quad 
\left\{
  \begin{aligned}
  &\text{for $e'\in \Sn$ with}\\
  &\text{$|e-e'|<\rho_0$ and $x \in \Sigma_1(e')$.}
  \end{aligned}
\right.
\end{equation}
 Again by $(E\chi)$, we then find 
$\rho_1 \in (0,\rho_0)$, depending only on $B,$ $d$, $k$, $M,$ and $\chi,$ such that 
\begin{equation}\label{eq2:hopf}
\nabla v^{e'}(x,1)\cdot e'\geq \frac{C_1}{2} \quad 
\left\{
  \begin{aligned}
  &\text{for $e'\in \Sn$ and $x\in \overline{B}$ }\\
  &\text{ with $|e-e'|<\rho_0$ and $|x \cdot e'|\le \rho_1.$}    
  \end{aligned}
\right.
\end{equation}
By Lemma~\ref{hopf:lemma:Neumann}, there is some 
$\eta_1>0$ which only depends on $B,$ $d$, $k$, $M,$ and $\chi,$ such that
\begin{align*}
 v^e(x,1)\geq \eta_1 \qquad\text{for $x \in \overline{B(e)}$ with $x
   \cdot e \ge \frac{\rho_1}{2}$.} 
\end{align*}
Again by $(E\chi),$ we may fix $\rho\in(0,\rho_1)$, depending only on $B,$ $d$, $k$, $M$ and $\chi,$ such that for all $e'\in \Sn$ with $|e-e'|<\rho,$
\begin{align}
v^{e'}(x,1) \geq \frac{\eta_1}{2}\qquad 
\text{for $x \in \overline{B(e')}$ with $x
   \cdot e' \ge \frac{\rho_1}{2}$.} 
\label{eq3:hopf_2}
\end{align}
For fixed $e'\in \Sn$ with $|e-e'|<\rho$, \eqref{eq2:hopf} ensures that 
$$
v^{e'}(x,1) = v(x,1) - v(\sigma_{e'}(x),1) >0 \qquad
\text{for $x \in B(e')$ with $x
   \cdot e' \le \frac{\rho_1}{2}$.} 
$$
Combining this with (\ref{eq3:hopf_2}), we find that 
$$
v^{e'}(x,1)  >0 \qquad
\text{for $x \in B(e')$,}
$$
as claimed. 
\end{proof}

\section{The scalar Neumann problem}\label{results:scalar:equations}

This section is devoted to the proof of Theorem \ref{main:theorem:scalar}. Let $u\in
C^{2,1}(\overline{B}\times(0,\infty))\cap
C(\overline{B}\times[0,\infty))$ be a (possibly
sign changing) solution of \eqref{model:scalar} such that the hypothesis
(H1)-(H4) and (\ref{eq:15}) of Theorem \ref{main:theorem:scalar} are fulfilled. We first note that 
\begin{equation*}
\begin{aligned}
u_t-\mu(|x|,t)\Delta u -c(x,t) u&=f(t,|x|,0) &&\qquad \text{in $B
  \times (0,\infty),$}\\
\partial_\nu u&=0 && \qquad \text{on $\partial B \times (0,\infty)$}\\
\end{aligned}
\end{equation*}
with 
$$
c(x,t):= \left \{
  \begin{aligned}
  &\frac{f(t,|x|,u(x,t))-f(t,|x|,0)}{u(x,t)},&&\quad \text{ if } u(x,t) \not=0,\\
  &0,&&\quad \text{ if } u(x,t)=0    
  \end{aligned}
\right.
$$
for $x \in \overline B$, $t>0$. By (H1) and (\ref{eq:15}) we have $c\in L^\infty(B\times(0,\infty))$, and thus (H2) and Lemma \ref{regularity} imply that the
functions 
\begin{equation}
  \label{eq:6}
\overline {B} \times [0,1] \to \R, \qquad (x,t) \mapsto
u(x,\tau+t),\qquad \tau \ge 1  
\end{equation}
and $\overline {B} \times [0,1] \to \R^N,$ $(x,t) \mapsto
\nabla u(x,\tau+t)$, $\tau \ge 1$ are uniformly equicontinuous. Hence
there exists a function
$\chi:[0,\sqrt{1+\operatorname{diam}(B)^2}\,] \to [0,\infty)$ with
$\lim \limits_{\vartheta \to 0}\chi(\vartheta)=0$ and such
that $(E\chi)$ of Lemma~\ref{perturbationlemma} holds for all of the functions in (\ref{eq:6}). Next, we set 
$$
u^e(x,t):=u(x,t)-u(\sigma_e(x),t) \qquad \text{for $x \in \overline B,\ 
  t >0,$ and $e\in \Sn.$}
$$
We wish to apply Corollary~\ref{sec:symm-char} to the sets $\cU:=
\omega(u)$ and  
\begin{align*}
\cN:=\{e\in \Sn \mid \exists \ \  T>0 \text{ such that } u^e(x,t)>0 \text{ for all }x\in B(e), \ t>T\} 
\end{align*}
With $\cM_\cU$ defined as in (\ref{eq:5}), it is obvious that $\cN
\subset \cM_{\cU}.$ We note that the function $u^e$ satisfies
\begin{equation*}
\begin{aligned}
u^e_t-\mu(|x|,t)\Delta u^e &=c^e(x,t) u^e && \qquad \text{in $B(e)
  \times (0,\infty)$},\\
\frac{\partial u^e}{\partial \nu} &= 0 &&\qquad \text{on $\Sigma_2(e)
  \times (0,\infty)$},\\
u^e &= 0 &&\qquad \text{on $\Sigma_1(e) \times (0,\infty)$},\\
 \end{aligned}
\end{equation*}
with $\Sigma_i(e)$ as defined in (\ref{eq:11}) and
$$
c^e(x,t):=
\left \{
  \begin{aligned}
   &\frac{f(t,|x|,u(x,t))-f(t,|x|,u(\sigma_e(x),t))}{u^e(x,t)},&&\quad \text{ if } u^e(x,t)
   \not=0,\\
   &0, &&\quad \text{ if } u^e(x,t)=0.   
  \end{aligned}
\right.
$$
By (H1), there exists $M>0$ with 
$$
\|c^e\|_{L^\infty(B\times(0,\infty))} \le M \qquad \text{for all $e
  \in \Sn$.}
$$
Moreover, by making $M$ larger if necessary and using (H3), we may
also assume that 
$$
\frac{1}{M} \le \mu(|x|,t) \le M \qquad \text{for all $x \in B$, $t>0$.}
$$
By (H4), there exists $\tilde e\in
\Sn$ such that $u^{\tilde  e}(\cdot,0)\ge 0$, $u^{\tilde
  e}(\cdot,0)\not \equiv 0$ on $B(\tilde  e)$ and thus $u^{\tilde
  e}(x,t)>0$ in $B(\tilde  e)\times(0,\infty)$ by Lemma
\ref{hopf:lemma:Neumann}, so that $\tilde e \in \cN$. Moreover, it easily follows from
Lemmas~\ref{hopf:lemma:Neumann} and \ref{perturbationlemma} that $\cN$ is a relatively open subset of
$\Sn$. By Corollary~\ref{sec:symm-char}, it therefore only remains to prove that $z
\le z\circ \sigma_e$ in $B(e)$ for every $z \in \omega(u)$ and $e \in \partial \cN$.  We argue by contradiction. Assume there is $\hat e \in \partial \cN$ and $z \in \omega(u)$ such that $z
\not \le z\circ \sigma_{\hat e}$ in $B(\hat e)$. Define
\begin{align*}
 z^e:\overline{B}\to\R \qquad by \qquad z^e(x):= z(x)-z(\sigma_e(x))
\end{align*}
for $e\in\Sn.$  Then there exist constants $d,k>0$ such
that 
 \begin{equation*}
 \sup\{ z^{\hat e}(x)\::\: x \in B,\: x \cdot \hat e \ge d \} > k
 \end{equation*}
We now let $\rho>0$ be given by Lemma~\ref{perturbationlemma}
corresponding to the choices of $d$, $k$, $M$ and $\chi$ made above. 
By continuity and since $\hat e \in \partial \cN$, there exists $e \in
\cN$ such that 
\begin{align}\label{rho1}
|e-\hat e|<\rho
\end{align}
and 
 \begin{equation}
\label{eq:8}
 \sup\{ z^{e}(x)\::\: x \in B,\: x \cdot e \ge d \} > k
\end{equation}
Let $(t_n)_{n} \subset (0,\infty)$ be a sequence with $t_n \to \infty$
and $u(t_n,\cdot) \to z$ in $L^\infty(\overline B)$. By (\ref{eq:8}),
there exists $n_0 \in \mathbb N$ such that 
 \begin{equation*}
 \sup\{u^{e}(t_n,x)\::\: x \in B,\: x \cdot e \ge d \} > k
\qquad \text{for all $n \ge n_0$.}
 \end{equation*}
Moreover, by the definition of $\cN$ there exists $T>0$ such that
$u^{e}(\cdot,t)>0$ in $B(e)$ for $t \ge T$. Next, fixing $n \in \mathbb N$ such that 
$t_n \ge \max\{T+\frac{1}{4},t_{n_0}\}$ and applying Lemma~\ref{perturbationlemma}
to the function
$$
\overline{B} \times [0,1]\to \R,\qquad (x,t) \mapsto u(x,t_n-\frac{1}{4}+t),
$$
we find, using (\ref{rho1}), that $u^{\hat e}(x,t_n+\frac{3}{4})>0$ for all $x\in
B(\hat e)$. Hence $\hat e \in \cN$. Since $\cN$ is relatively open in
$\Sn$, this contradicts the fact that $\hat e \in \partial \cN$. The proof of
Theorem~\ref{main:theorem:scalar} is thus finished.

\section{Proof of the main result for competitive systems}
\label{normalization:argument}

In this section we will complete the proof of
Theorem~\ref{main:theorem:neumann}. For the remainder of this section,
let $u_1,u_2\in C^{2,1}(\overline{B}\times(0,\infty))\cap
C(\overline{B}\times[0,\infty))$ be functions such that $u=(u_1,u_2)$ solves
(\ref{model:competitive:neumann}) and such that assumptions
(h0)--(h3), (\ref{eq:12}) from the introduction are fulfilled. A key ingredient of
the proof is the following quotient estimate which compares the values
of the components of $u$ at different times. Similar estimates were
obtained by J. H\'{u}ska, P. Pol\'{a}\v{c}ik, and M. V. Safonov in
\cite[Corollary 3.10]{huska:polacik:safonov} for positive solutions of scalar
parabolic Dirichlet problems. We point out that the Neumann boundary
conditions on $\partial B$ allow to obtain a stronger result in the
present setting with a much simpler proof. 
In the following, for matters of simplicity, we sometimes omit the
arguments $(x,t)$ and $(|x|,t)$.

\begin{lemma} \label{normalization:lemma} 
There exists a constant $\eta>1$ such that
\begin{equation*}
\frac{1}{\eta}\le
\frac{u_i}{\|u_i(\cdot,\tau)\|_{L^\infty(B)}} \le \eta \ \ \ 
 \text{in $B\times[\tau -3, \tau+3]$}
\end{equation*}
for all $\tau\geq 5$ and $i=1,2.$
\end{lemma}
\begin{proof}
We only prove the estimate for $i=1$, the proof for $i=2$ is the
same. For simplicity, we write $u$ in place of $u_1$, and we
note that 
\begin{equation*}
u_t-\mu_1 \Delta u = c\, u \qquad \text{in $B
  \times (0,\infty)$}
\end{equation*}
with 
$$
c(x,t):= \alpha_1(t,|x|)u_2(x,t) + \left\{
  \begin{aligned}
   &\frac{f_1(t,|x|,u(x,t))}{u(x,t)},&&\quad \text{ if } u(x,t) \not=0,\\
   &0,&&\quad \text{ if } u(x,t)=0. 
  \end{aligned}
\right.
$$
By (h1), (h3), and (\ref{eq:12}),
we have that $c \in L^\infty(B\times(0,\infty))$. Let $\tilde u$ denote the extension of $u$ to $\widetilde B$ as
defined in (\ref{eq:16}). Then Lemma~\ref{extension:lemma}
implies that $\tilde u$ is a strong solution of 
\begin{equation*}
\begin{aligned}
(\tilde u)_t- \tilde \mu\, \Delta \tilde u-\tilde
b\,\partial_r\tilde u =\tilde c\, \tilde u \qquad \text{in
  $\widetilde B \times (0,\infty)$.} 
\end{aligned}
\end{equation*}
Here $\tilde \mu, \tilde b \in L^\infty(\widetilde B \times(0,\infty))$ are defined as in
Lemma~\ref{extension:lemma} with $\mu$ replaced by $\mu_1$, and
$\tilde c \in L^\infty(\widetilde B \times(0,\infty))$ is defined by 
\begin{equation*}
\tilde c(x,t):= \begin{cases}
c(x,t), &\quad  x\in B,\ t\in (0,\infty),\\\vspace{.1cm}
c(\hat x,t), &\quad  x\in \widetilde B \setminus B,\ t\in (0,\infty).
\end{cases}
\end{equation*}  
We also note that $\inf \limits_{B \times (0,\infty)} \tilde \mu >0$
as a consequence of $(h2)$. Next, we fix $\tau\geq 5,$ and we apply the Harnack inequality for
strong solutions given in \cite[Lemma 3.5]{polacik} (with
$p=\infty$, $U=\widetilde B,$ $D=B,$ and $v=\tilde u$). 
The application yields $\kappa_1>0$ independent of $\tau$ such that
\begin{align}\label{1}
\inf_{B\times(\tau-3,\tau+3)}u\geq \kappa_1 \|u(\cdot,\tau-4)\|_{L^\infty(B)},
\end{align}
since $\tilde u$ coincides with $u$ on $B \times (0,\infty)$. Moreover, by the maximum principle (see for example \cite[Lemma
7.1]{lieberman}) and the uniform bounds on the coefficients, there
exists $\kappa_2>\kappa_1$ independent of $\tau$ such that 
\begin{equation}\label{2}
\|u(\cdot,s)\|_{L^\infty(B)}\leq \kappa_2 \|u(\cdot,\tau-4)\|_{L^\infty(B)}
\quad \text{for $s \in [\tau-3,\tau+3]$.}
\end{equation}
Let $x\in B$ and $t\in[\tau-3,\tau+3].$ Then, by \eqref{1} and \eqref{2},
\begin{align*}
 \frac{u(x,t)}{\|u(\cdot,\tau)\|_{L^\infty(B)}}\geq \frac{\kappa_1\|u(\cdot,\tau-4)\|_{L^\infty(B)}}{\|u(\cdot,\tau)\|_{L^\infty(B)}}\geq \frac{\kappa_1}{\kappa_2},
\end{align*}
and 
\begin{align*}
 \frac{u(x,t)}{\|u(\cdot,\tau)\|_{L^\infty(B)}}
 \leq\frac{\kappa_2\|u(\cdot,\tau-4)\|_{L^\infty(B)}}{\|u(\cdot,\tau)\|_{L^\infty(B)}}
 \leq \frac{\kappa_2}{\kappa_1}.
\end{align*}
Thus the claim follows with $\eta= \frac{\kappa_2}{\kappa_1}$.
\end{proof}

Next, we slightly change some notation used in previous
sections in order to deal with competitive
systems of two equations. For $e\in\Sn,$ a radial domain $B\subset \mathbb R^N$,
$I\subset \R$ and a pair $v=(v_1,v_2)$ of functions $v_i: \overline{B}\times I\to\R,$ $i=1,2$,
we set 
\begin{equation}\label{difference:function:definition1}
\begin{aligned}
 v_1^e(x,t)&:=v_1(x,t)-v_1(\sigma_e(x),t),\ x\in \overline{B},\ t>0,\\
 v_2^e(x,t)&:=v_2(\sigma_e(x),t)-v_2(x,t),\ x\in \overline{B},\ t>0,\\
\end{aligned}
\end{equation}
The same notation is used if the functions do not depend on time. More
precisely, for a pair $z=(z_1,z_2)$ of functions $z_i: \overline{B} \to\R,$ $i=1,2$,
we set 
\begin{equation}\label{difference:function:definition2}
\begin{aligned}
 z_1^e(x)&:=z_1(x)-z_1(\sigma_e(x)),\ x\in \overline{B},\\
 z_2^e(x)&:=z_2(\sigma_e(x))-z_2(x),\ x\in \overline{B}.
\end{aligned}
\end{equation}
Since $u=(u_1,u_2)$ solves (\ref{model:competitive:neumann}), for fixed $e \in
\Sn$ we have 
\begin{equation*}
\begin{aligned}
(u_1^e)_t-\mu_1\Delta u_1^e - \hat c_1^e(x,t)u_1^e=&\alpha_1
[\hat u_1 \hat u_2 -u_1u_2]= \alpha_1 [u_1 u_2^e - \hat u_2  u_1^e] 
,\\
(u_2^e)_t-\mu_2\Delta u_2^e - \hat c_2^e(x,t)u_2^e=&\alpha_2 [u_1
u_2 -\hat u_1 \hat u_2]= \alpha_2 [u_2 u_1^e  -\hat u_1 u_2^e]   
\end{aligned}
\end{equation*}
in $B \times (0,\infty)$ with $\hat u_i(x,t):= u_i(\sigma_e(x),t))$ and 
\begin{equation*}
 \hat c_i^e(x,t):= 
\left\{
  \begin{aligned}
&\frac{f_i(t,|x|,u_i(x,t))-f_i(t,|x|,u_i(\sigma_e(x),t))}{u_i(x,t)-u_i(\sigma_e(x),t)},&&\quad \text{ if }
u_i^e(x,t) \not=0,\\
&0,&&\quad \text{ if } u_i^e(x,t)=0    
  \end{aligned}
\right.
\end{equation*}
for $i=1,2$. Setting 
\begin{align*}
c^e_1(x,t)&:=\hat c_1^e(x,t)- \alpha_1(|x|,t) u_2(\sigma_e(x),t)\\
c^e_2(x,t)&:=\hat c_2^e(x,t)- \alpha_2(|x|,t) u_1(\sigma_e(x),t)  
\end{align*}
for $x\in B$, $t>0$, we thus obtain the system 
\begin{equation}\label{linear:neumann}
\begin{aligned}
(u^e_1)_t-\mu_1\Delta u_1^e -c^e_1 u_1^e&= \alpha_1 u_1 u_2^e\\
(u^e_2)_t-\mu_2\Delta u_2^e -c^e_2 u_2^e & = \alpha_2 u_2 u_1^e
\end{aligned}
\qquad \text{in $B(e) \times (0,\infty)$}
\end{equation}
together with the boundary conditions 
\begin{equation}\label{linear:neumann-boundary}
\frac{\partial u^e_i}{\partial \nu}= 0 \quad \text{on $\Sigma_2(e)
  \times (0,\infty)$},\qquad 
u^e_i= 0 \quad \text{on $\Sigma_1(e)
  \times (0,\infty)$,}
\end{equation}
where the sets $\Sigma_i(e)$ are given as in (\ref{eq:11}) for
$i=1,2$. As a consequence of (h1),(h3), and (\ref{eq:12}), we have 
\begin{equation}\label{coefficient:estimates}
 \|c^e_1\|_{L^\infty(B \times (0,\infty))} \le M \quad \text{and}\quad \|c^e_2\|_{L^\infty(B \times (0,\infty))} \le M \qquad \text{for all $e\in \Sn$}
\end{equation}
with some constant $M>0$. Moreover, by making $M$ larger if necessary and using (h2), we may
also assume that 
\begin{equation}
  \label{eq:18}
\frac{1}{M} \le \mu_i(|x|,t) \le M \qquad \text{for $x \in B$, $t>0,$ and $i=1,2$.}
\end{equation}
We note that, by (h3) and since $u_1,u_2\geq 0$ in $B \times (0,\infty),$ system
\eqref{linear:neumann} is a (weakly coupled)
cooperative parabolic system. For these systems a variety of
estimates are available (see for example \cite{protter} and
\cite{polacik:systems}). In particular, Lemma~\ref{hopf:lemma:Neumann:systems}
can be applied to study the boundary value problem (\ref{linear:neumann}),~(\ref{linear:neumann-boundary}).

To prove Theorem~\ref{main:theorem:neumann}, we wish to apply
Corollary~\ref{sec:symm-char} to the sets 
\begin{equation}
  \label{eq:7}
\cU:=\omega(u_1) \cup -\omega(u_2) = \{z_1,-z_2\::\: z \in \omega(u)\}
\end{equation}
and 
\begin{equation}
  \label{eq:21}
\cN:=\{e\in \Sn\::\: \text{$\exists\ T>0$ s.t. $u^e_i> 0$ in $B(e)
  \times [T,\infty)$ for $i=1,2$} \}.
\end{equation}
Note that the equality in (\ref{eq:7}) is a consequence of
(\ref{eq:3}). In this case the associated set $\cM_{\cU}$,
defined in (\ref{eq:5}), can also be written as
$$
\cM_\cU=\{e\in \Sn \::\:
z_i^e \ge 0 \quad\text{in $B(e)$ for all $z \in \omega(u)$, $i=1,2$}\}.
$$
Thus we obviously have $\cN\subset \cM_{\cU}.$ Moreover, for $e\in
\Sn$ as in (h0), we have
\begin{align*}
u_i^e(\cdot,0)\geq 0,\; u_i^e(\cdot,0)\not\equiv 0 \qquad \text{in
  $B(e)$ for $i=1,2$,}
\end{align*}
Lemma~\ref{hopf:lemma:Neumann:systems} then implies that $u_i^e> 0$ in
  $B(e) \times (0,\infty)$ for $i=1,2$, so that $e \in \cN$ and
  thus $\cN$ is nonempty. We also note the following.

\begin{lemma} \label{M:open}
$\cN$ is relatively open in $\Sn$. 
\end{lemma}
\begin{proof} Let $e\in \cN.$ Then $(u_1^e,u_2^e)$ is a solution of \eqref{linear:neumann}, and there is $T>0$ such that $u^e_1$ and $u^e_2$ are positive in $B(e)\times(T,\infty).$ Thus
\begin{align*}
(u_1^e)_t-\mu_1\Delta u_1^e- c_1^e u_1^e &=\alpha_1 u_1 u_2^e\geq 0,\ \ x\in B(e),\ t>T,\\
(u_2^e)_t-\mu_2\Delta u_2^e- c_2^e u_2^e &=\alpha_2 u_2 u^e_1\geq 0,\ \ x\in B(e),\ t>T,
\end{align*}
since $\alpha_1$ and $\alpha_2$ are non-negative by hypothesis (h3). 
Applying Lemma \ref{perturbationlemma} and Remark~\ref{sec:harnack-hopf-type} 
to the functions 
$$
\overline {B} \times [0,1] \to \R, \qquad (x,t) \mapsto u_i(x,T+t),\qquad
i=1,2,
$$
we find that there exists $\rho>0$ such that $u_i^{e'}(\cdot,T+1)>0$
in $B(e')$ for $e'\in\Sn$ with $|e'-e|<\rho.$ Hence, by Lemma
\ref{hopf:lemma:Neumann:systems}, $e'\in \cN$ for $e'\in\Sn$ with
$|e'-e|<\rho$, and thus $\cN$ is open. 
\end{proof}

In order to apply Corollary~\ref{sec:symm-char}, it now suffices to
prove the following. 

\begin{lemma}\label{normalization:lemma:2}
For every $e\in \partial \cN$ and every $z \in \omega(u)$ we have $z^e_1 \equiv z^e_2 \equiv 0$ in $B(e)$.
\end{lemma}

\begin{proof}
Let $z=(z_1,z_2) \in\omega(u)$, and consider an increasing sequence
$t_n\to\infty$ with $t_1>5$ and such that $u_i(\cdot,t_n)\to z_i$
uniformly in $\overline B$ for $i=1,2.$ We will only show that $z_2^e
\equiv 0$ in $B(e)$ for all $e \in \partial \cN$, since the same
argument shows that $z_1^e\equiv 0$ in $B(e)$ for all $e
\in \partial \cN$. Since, as noted in Remark
\ref{equicontinuity}, $u_2$ and its first derivatives satisfy the
H{\"o}lder condition (\ref{equicontinuity:h}), there exists a function
$\chi:[0,\sqrt{1+\operatorname{diam(B)}^2}\,] \to [0,\infty)$ with $\lim \limits_{\vartheta \to 0}\chi(\vartheta)=0$ and such
that the equicontinuity condition $(E\chi)$ of
Lemma~\ref{perturbationlemma} holds for all of the functions 
\begin{equation}
  \label{eq:13}
\overline B \times [0,1] \to \R, \qquad (x,t) \mapsto
u_2(x,\tau +t),\qquad  \tau \ge 1.
\end{equation}
Arguing by contradiction, we now assume that
$z_2^{\hat e}\not\equiv 0$  in $B(\hat e)$ for some $\hat e
\in \partial \cN$. By the equicontinuity of the functions in
(\ref{eq:13}), there are $\zeta \in(0,\frac{1}{4})$, a nonempty open
subset $\Omega \subset \subset B(\hat e),$ and $k_1>0$ such that, after passing to
a subsequence, 
\begin{equation}\label{contradiction:hypothesis}
u_2^{\hat e} \ge k_1 \quad \text{on $\Omega \times
  [t_n-\zeta,t_n+\zeta]\quad$ for all $n \in \mathbb N$.}
\end{equation}
We now apply a normalization procedure for $u_1$, since we cannot exclude the possibility
that $u_1(\cdot,t_n) \to 0$ as $n \to \infty$. Define, for $n\in\mathbb N,$
\begin{equation*}
I_n:=[t_n-2,t_n+2] \subset \R,\qquad
\beta_n:=\|u_1(\cdot,t_n)\|_{L^\infty(B)}
\end{equation*}
and the functions 
$$
v_n : \overline{B} \times I_n \to \R, \qquad v_n(x,t)= \frac{u_1(x,t)}{\beta_n}.
$$
By Lemma \ref{normalization:lemma}, there exists $\eta>1$ such that 
\begin{equation}
 \frac{1}{\eta} \:\le\: v_n \: \leq\: \eta \quad \text{in $B\times I_n$}\qquad
 \text{for all 
   $n\in\mathbb N.$} \label{eta}
\end{equation}
Moreover, we have 
\begin{align}\label{hoeldercontinuos:1}
\sup_{\genfrac{}{}{0pt}{}{\scriptstyle{x,\bar x\in
        \overline{B},\, t,\bar t\in[s,s+1],}}{{\scriptstyle{x \not= \bar x,\, t \not=
        \bar t,\, s\in[-1,1]}}}
} \frac{|v_n(x,t_n+t)-v_n(\bar x,t_n+\bar t)|}{|x-\bar x|^\gamma+|t-\bar
  t|^{\frac{\gamma}{2}}}< K,
\end{align}
and
\begin{align*}
\sup_{\genfrac{}{}{0pt}{}{\scriptstyle{x,\bar x\in
        \overline{B},\, t,\bar t\in[s,s+1],}}{{\scriptstyle{x \not= \bar x,\, t \not=
        \bar t,\, s\in[-1,1]}}}
} \frac{|\nabla v_n(x,t_n+t)-\nabla v_n(\bar x,t_n+\bar t)|}{|x-\bar x|^\gamma+|t-\bar
  t|^{\frac{\gamma}{2}}}< K
\end{align*}
for all $n\in\mathbb N$ with positive constants $\gamma$ and $K$. This follows
from Lemma~\ref{regularity} and the fact that $v_n$ satisfies
\begin{equation*}
\begin{aligned}
(v_n)_t-\mu_1\Delta v_n &= c v_n - \alpha_1 v_n u_2 &&\qquad \text{in $B \times I_n$},\\
\partial_\nu v_n&=0 &&\qquad \text{on $\partial B \times I_n$}
\end{aligned}
\end{equation*}
with 
$$
c \in L^\infty(B\times (0,\infty)),\qquad
c(x,t):=
\left\{
  \begin{aligned}
  &\frac{f_1(t,|x|,u_1(t,x))}{u_1(t,x)},&&\quad \text{ if }u_1(t,x) \not=0,\\
  &0,&&\quad \text{ if }u_1(t,x)=0.    
  \end{aligned}
\right.
$$ 
As a consequence, by adjusting the function $\chi$ above, we may also assume that all of the functions 
\begin{equation*}
\overline B \times [0,1] \to \R, \; (x,t) \mapsto v_n(x,\tau+t),\quad \text{$|t_n - \tau|\le 1$ for some $n \in \mathbb N$} 
\end{equation*}
satisfy the equicontinuity condition $(E\chi)$ of
Lemma~\ref{perturbationlemma}. For $e \in \Sn$, $n \in \mathbb N$ we also consider 
$$
v_n^e : \overline{B(e)} \times I_n \to \R,\qquad v_n^e(x,t):= v_n(x,t)-v_n(\sigma_e(x),t),
$$
and we note that 
\begin{equation}\label{normalized:equations}
\begin{aligned}
(v_n^e)_t-\mu_1\Delta v_n^e-c_1^e v_n^e &=\alpha_1 v_n u_2^e &&\qquad
\text{in $B(e) \times I_n$,}\\
(u_2^e)_t-\mu_2\Delta u_2^e- c_2^e u_2^e &=\alpha_2 \beta_n u_2 v^e_n
&&\qquad
\text{in $B(e) \times I_n$,}\\
\partial_\nu v_n^e=\partial_\nu u_2^e &= 0 &&\qquad \text{on
  $\Sigma_2(e) \times I_n$,}\\
v^e_n(x,t)=u_2^e(x,t) &= 0 &&\qquad \text{on $\Sigma_1(e) \times I_n$}
\end{aligned}
\end{equation}
with $\Sigma_i(e)$ as defined in (\ref{eq:11}). We now distinguish two cases. 

\begin{align*}
\hspace{-4cm}\text{ \underline{Case 1:} }\ \ \ \limsup_{n\to\infty}\|v_n^{\hat e}\|_{L^\infty(B(\hat e)\times [t_n-\zeta,t_n+\zeta])}>0. 
\end{align*}
In this case, by \eqref{hoeldercontinuos:1},
there are $d \in (0,1)$, $k_2>0,$ and $t^*\in[-\zeta,\zeta]$ such
that, after passing to a subsequence, 
$$
\sup \{v^{\hat e}_n(x,t_n + t^*) : x \in B({\hat e}),\: x\cdot {\hat e} \ge d\} \ge k_2 \qquad
\text{for $n \in \mathbb N$.}
$$
Without loss, we may assume that $d< \min \{x \cdot {\hat e}\::\: x \in
\Omega\}$, so that also 
$$
\sup \{u^{\hat e}_2(x,t_n + t^*) : x \in B({\hat e}),\: x\cdot {\hat e} \ge d\} \ge k_1 \qquad
\text{for $n \in \mathbb N$}
$$
by (\ref{contradiction:hypothesis}). Next, let $k:= \frac{1}{2}\min\{k_1,k_2\},$
and let $\rho>0$ be the constant given by Lemma~\ref{perturbationlemma} for
$M$ satisfying (\ref{coefficient:estimates}),~(\ref{eq:18}) and $d$,
$k$, $\chi$ as chosen above. Since $\hat e
\in \partial \cN$, there exists $e \in {\cal N}$ such that 
$|e-\hat e|<\frac{\rho}{2}$ 
and, by equicontinuity,  
\begin{align*}
&\sup \{v^e_n(x,t_n + t^*) : x \in B(e),\: x\cdot e \ge d\} \ge k,\\
&\sup \{u^e_2(x,t_n + t^*) : x \in B(e),\: x\cdot e \ge d\} \ge k
\end{align*}
for all $n \in \mathbb N$. Since $e\in{\cal N}$ we can fix $n\in\mathbb N$ such that
\begin{align*}
 v_n^e(x,t_n+t^*-\frac{1}{4})\geq 0,\quad u_2^e(x,t_n+t^*-\frac{1}{4})\geq 0 \qquad \text{for all }x\in B(e).
\end{align*}
Then applying Lemma~\ref{perturbationlemma} to the functions 
$$
\overline B \times [0,1] \to \R, \qquad (x,t) \mapsto
u_2(x,t_n +t^*-\frac{1}{4} +t),\qquad (x,t) \mapsto
v_n(x,t_n +t^*-\frac{1}{4} +t),
$$
we conclude that 
$$
u_{2}^{\bar e}(\cdot,t_{n}+t^*+\frac{3}{4})>0 \quad
\text{and}\quad v_{n}^{\bar e}(\cdot,t_{n}+t^*+\frac{3}{4})>0
\qquad \text{in $B(\bar e)$}
$$
for all $\bar e\in \Sn$ with $|\bar e-e|<\rho$, and thus in
particular for $\bar e = \hat e$. This yields $u_i^{\hat
  e}(\cdot,t_{n}+t^*+\frac{3}{4})>0$ in $B(\hat e)$ for $i=1,2$, and
thus $\hat e \in \cN$   
by Lemma \ref{hopf:lemma:Neumann:systems}. Since $\cN \subset \Sn$ is
relatively open by Lemma \ref{M:open}, this contradicts the
hypothesis that $\hat e \in \partial \cN$.

\begin{align}\label{normalization:goes:to:zero}
\hspace{-4cm}\text{ \underline{Case 2:} }\ \ \ \lim_{n\to\infty}\|v_n^{\hat e}\|_{L^\infty(B({\hat e})\times [t_n-\zeta,t_n+\zeta])}=0. 
\end{align}
In this case we fix a
nonnegative function $\varphi\in C_c^\infty(B({\hat e}) \times
(-\zeta,\zeta))$ with $\varphi \equiv 1$ on $\Omega \times
(-\frac{\zeta}{2},\frac{\zeta}{2})$. Moreover, we let 
$$
\Omega_n:= B({\hat e})
\times (t_n-\zeta,t_n+\zeta)\qquad \text{and}\qquad \varphi_n \in
C_c^\infty(\Omega_n),\quad \varphi_n(x,t):= \varphi(x,t_n+t).
$$
Setting
$(u_2^{\hat e})^+:=\max\{u_2^{\hat e},0\}$ and $(u_2^{\hat e})^-:=-\min\{u_2^{\hat e},0\},$ we
find by (h3), (\ref{contradiction:hypothesis}) and \eqref{eta} that 
\begin{align*}
A_n:= \int_{\Omega_n} &\alpha_1 v_n u_2^{\hat e}\varphi_n d(x,t)
=\int_{\Omega_n} \alpha_1 v_n [(u_2^{\hat e})^+-(u_2^{\hat e})^-]\varphi_n d(x,t)\nonumber\\
&\geq \frac{\alpha_*}{\eta} \int_{\Omega_n} (u_2^{\hat e})^+\varphi_n
d(x,t)- \alpha^* \eta \,\|(u_2^{\hat e})^-\|_{L^\infty(\Omega_n)}\,\|\varphi_n\|_{L^1(\Omega_n)},\nonumber\\
&\geq \frac{\alpha_*}{\eta} k_1 |\Omega| \zeta \;-\;\alpha^* \eta\, 
\|(u_2^{\hat e})^-\|_{L^\infty(\Omega_n)}\,\|\varphi\|_{L^1(\Omega \times
(-\zeta,\zeta))},
\end{align*}
for $n \in \mathbb N$, whereas $\lim \limits_{n \to \infty}\|(u_2^{\hat
  e})^-\|_{L^\infty(\Omega_n)}=0$ since ${\hat e}\in \partial \cN.$
Hence $\liminf \limits_{n \to
  \infty} A_n>0$. On the other hand, integrating by parts, we have by \eqref{normalized:equations} that
\begin{align*}
A_n &=\int_{\Omega_n}\!\![(v_n^{\hat e})_t-\mu_1\Delta v_n^{\hat e}- c_1^{\hat e} v_n^{\hat e}]\varphi_n d(x,t)\\ &=-\int_{\Omega_n}\!\![v_n^{\hat e}(\varphi_n)_t +v_n^{\hat e} \Delta (\mu_1\varphi_n) + c_1^{\hat e} v_n^{\hat e}\varphi_n ]d(x,t)\\
&\leq \|v^{\hat e}_n\|_{L^\infty(\Omega_n)} \int_{\Omega_n}\Bigl(|(\varphi_n)_t
|+ |\Delta(\mu_1 \varphi_n)| + M \varphi_n\Bigr) d(x,t)
\end{align*}
for $n \in \mathbb N$. Invoking (h2) and \eqref{normalization:goes:to:zero}, we conclude that
$\limsup \limits_{n\to\infty} A_n\le0.$ So we have
obtained a contradiction again, and thus the claim follows.
\end{proof}

\begin{proof}[Proof of Theorem~\ref{main:theorem:neumann} (completed)] 
By Lemmas~\ref{M:open} and \ref{normalization:lemma:2} and the remarks
before Lemma~\ref{M:open}, the assumptions of
Corollary~\ref{sec:symm-char} are satisfied with $\cU$ and $\cN$ as
defined in (\ref{eq:7}) and (\ref{eq:21}). Consequently, there exists
$p \in \Sn$ such that every $z \in \cU$ is foliated Schwarz symmetric
with respect to $p$. By definition of $\cU$, this implies that every
$z=(z_1,z_2) \in \omega(u)$ has the property that $z_1$ is foliated
Schwarz symmetric with respect to $p$ and $z_2$ is foliated Schwarz
symmetric with respect to $-p$.
\end{proof}

\section{The cooperative case and other problems}\label{sec:other:problems}

In this section we first complete the 
\begin{proof}[Proof of Theorem~\ref{main:theorem:neumann-cooperative}]
Let $u_1,u_2\in C^{2,1}(\overline{B}\times(0,\infty))\cap
C(\overline{B}\times[0,\infty))$ be functions such that $u=(u_1,u_2)$ solves
\eqref{model:cooperative:neumann}, and suppose that $(h0)'$,
$(h1)$--$(h3)$ and (\ref{eq:12}) are satisfied. The proof is almost
exactly the same as the one of Theorem~\ref{main:theorem:neumann}
with only two changes. The first change concerns
the definitions of $v_2^e$ and $z_2^e$ in
(\ref{difference:function:definition1}) and
(\ref{difference:function:definition2}). More precisely, we now set $v_i^e(x,t)= v_i(x,t)-v_i(\sigma_e(x),t)$ and $z_i^e(x)=
z_i(x)-z_i(\sigma_e(x))$ for $i=1,2$. With this change, we again arrive at
the linearized system~(\ref{linear:neumann}). Considering now the sets
\begin{equation*}
\cU:=\omega(u_1) \cup \omega(u_2) = \{z_1,z_2\::\: z \in \omega(u)\}
\end{equation*}
in place of (\ref{eq:7}) and  
\begin{equation*}
\cN:=\{e\in \Sn\::\: \text{$\exists\ T>0$ s.t. $u^e_i> 0$ in $B(e)
  \times [T,\infty)$ for $i=1,2$} \},
\end{equation*}
we may now validate the assumptions of Corollary~\ref{sec:symm-char}
in exactly the same way as in Section~\ref{normalization:argument}. Hence
the proof is complete.
\end{proof}

\begin{remark}
(i) Note that both in Theorem~\ref{main:theorem:neumann} and in
Theorem~\ref{main:theorem:neumann-cooperative} we assume that the
components $u_i$ are non-negative, and this assumption is essential for the
cooperativity of the linearized system~(\ref{linear:neumann}). 
Without the sign restriction, systems~(\ref{model:competitive:neumann}) and
(\ref{model:cooperative:neumann}) arise from each other by
replacing $u_i$ by $-u_i$ for $i=1,2$ and adjusting $f$ accordingly.\\
(ii) As a further example, we wish to mention the cubic system
\begin{equation}\label{cubic:system}
 \begin{aligned}
(u_1)_t-\Delta u_1 &=\lambda_1u_1+\gamma_1u_1^3-\alpha_1u^{2}_2 u_1
&&\qquad \text{in $B \times (0,\infty)$,}\\
(u_2)_t-\Delta u_2 &=\lambda_2u_2+\gamma_2u_2^3-\alpha_2u_1^2 u_2&&\qquad \text{in $B \times (0,\infty)$,}\\
\partial_\nu u_1&=\partial_\nu u_2=0&&\qquad \text{on $\partial B
  \times (0,\infty)$},\\
u_i(x,0)&=u_{0,i}(x) \ge 0&& \qquad \text{for $x\in B$, $i=1,2$,}
\end{aligned}
\end{equation}
where $\lambda_i,\gamma_i,$ and $\alpha_i$ are positive constants. The
elliptic counterpart of this system is being studied
extensively due to its relevance in the study of binary mixtures of Bose-Einstein
condensates, see \cite{esry}. The asymptotic symmetry of uniformly
bounded classical solutions of this problem satisfying the initial
reflection inequality condition $(h0)$ can be characterized in
the same way as in Theorem~\ref{main:theorem:neumann}. To see
this, minor adjustments are needed in the proof of
Theorem~\ref{main:theorem:neumann} to deal with a slightly different
linearized system. Details will be given in
\cite{saldana-phd}. Symmetry aspects of the elliptic counterpart of
(\ref{cubic:system}) have been studied in \cite{weth:tavares}.\\ 
(iii) Our method breaks down if the coupling term has different signs
in the components, as e.g. in a predator-prey
type system 
\begin{equation*}
\begin{aligned}
 (u_1)_t-\mu_1\Delta u_1 &=f_1(t,|x|,u_1)+\alpha_1 u_1u_2&&\qquad
 \text{in $B \times (0,\infty)$},\\
 (u_2)_t-\mu_2\Delta u_2 &=f_2(t,|x|,u_2)-\alpha_2 u_1u_2&&\qquad 
\text{in $B \times (0,\infty)$}.
\end{aligned}
\end{equation*}
In this case, there seems to be no way to derive a cooperative
linearized system of the type (\ref{linear:neumann}) for difference functions related to hyperplane
reflections. The asymptotic shape of solutions for this system
(satisfying Dirichlet or Neumann boundary conditions) remains an
interesting open problem.\\
(iv) Consider general systems of the form  
\begin{equation}\label{general:cooperative:model}
\begin{aligned}
 (u_i)_t-\Delta u_i &=f_i(t,|x|,u)&&\qquad \text{in $B \times (0,\infty)$},\\
 \partial_{\nu} u_i&=0&&\qquad \text{on $\partial B \times (0,\infty)$},\\
\end{aligned}
\end{equation}
for $i=1,2$, where the nonlinearities $f_i:[0,\infty)\times I_B \times \mathbb R^2 \to \mathbb R$ are
locally Lipschitz in $u=(u_1,u_2)$ uniformly with respect to $r\in I_B$ and
$t>0$. We call (\ref{general:cooperative:model}) an {\em irreducible
cooperative system} if for every $m>0$ there is a constant $\sigma>0$
such that 
$$
\frac{\partial f_{i}(t,r,u)}{\partial u_j}
 \:  \geq \: \sigma\quad \left \{
\begin{aligned}
&\text{for every $i,j \in \{1,2\}$, $i \not=j$, $r\in I_B$, $t>0,$ $|u|\le m$}\\
&\text{such that the derivative exists.}  
\end{aligned}
\right.
$$
For this class of systems a symmetry result similar to
Theorem~\ref{main:theorem:neumann-cooperative} can be derived 
{\em even for sign changing solutions}, and in fact the proof is simpler. 
The precise statement and detailed arguments are given in \cite{saldana-phd},
while we only discuss the key aspects here. We first note that, for a given uniformly bounded
classical solution $u=(u_1,u_2)$ of (\ref{general:cooperative:model}) and $e \in \Sn$, we can use the
Hadamard formulas as in \cite{polacik:systems} to derive a cooperative
system for the functions $(x,t) \mapsto
u^e_i(x,t):=u_i(x,t)-u_i(\sigma_e(x),t)$. This system has the form 
\begin{align*}
 (u_i^e)_t-\Delta u_i^e = \sum_{j=1}^2 c^e_{ij} u^e_j \qquad \text{in $B(e)\times (0,\infty)$}
\end{align*}
with functions $c^e_{ij} \in L^\infty(B\times (0,\infty))$, $i,j=1,2$
such that 
$$
\inf_{B(e) \times (0,\infty)}c^e_{ij}>0 \qquad \text{for $i\neq
j$.}
$$
With the help of the latter property, one can prove that for every sequence of
positive times $t_n$ with $t_n \to \infty$ and every $e \in \Sn$ we have the equivalence 
$$
\lim_{n \to \infty}\|u_1^e(\cdot,t_n)\|_{L^\infty(B(e))}= 0 \qquad \Longleftrightarrow
\qquad \lim_{n \to \infty}\|u_2^e(\cdot,t_n)\|_{L^\infty(B(e))}= 0.
$$
As a consequence, semitrivial limit profiles $(z_1,0), (0,z_2) \in \omega(u)$
have the property that the nontrivial component must be a radial
function, and hence no normalization procedure as in
Section~\ref{normalization:argument} is needed to deal with these
profiles. This is the reason why the positivity of components is not
needed in this case. Details are given in \cite{saldana-phd}. Note that the cooperative system
(\ref{model:cooperative:neumann}) is {\em not} irreducible. \\
(v) The arguments and results for irreducible cooperative systems sketched in (iv) also apply to a corresponding system with $n\geq 3$ equations.  On the other hand, one may also consider cooperative systems of the form 
\begin{equation}\label{model:cooperative:neumann_1}
 (u_i)_t-\mu_i(|x|,t)\Delta u_i =f_i(t,|x|,u_i)+\sum_{\stackrel{j=1}{j \not=i}}^n \alpha_{ij}(|x|,t) u_iu_j, \quad i=1,\dots,n.
\end{equation}
with $n \ge 3$ equations which are not irreducible. Assume that $(h1)$ and $(h_2)$ hold for $f_i$ and $\mu_i$, $i=1,\dots,n$, and that $\alpha_{ij} \in L^\infty(I_B \times (0,\infty))$ are nonnegative functions for $i,j=1,\dots,n$, $i \not = j$.  It is then an open question which additional positivity assumptions on the coefficients $\alpha_{ij}$ are required for the corresponding generalization of Theorem~\ref{main:theorem:neumann-cooperative}. Similar arguments as in Section 5 apply in the case where $\alpha_{ij} \ge \alpha_*>0$ for $i,j=1,\dots,n$, $i \not = j$, but we do not think that this assumption is optimal.  We thank the referee for pointing out this question.
\end{remark}

\section{Appendix}

Here we show the existence of positive solutions of the
elliptic system~(\ref{Lotka:Volterra:system-elliptic}) without
foliated Schwarz symmetric components. More precisely, we have the
following result. 

\begin{theo}\label{thm:local:maxima:system} Let $k\in\mathbb N$. 
Then there exists $\eps, \lambda >0$ such that (\ref{Lotka:Volterra:system-elliptic})  admits a positive classical solution $(u_1,u_2)$ in $B:= B_\eps= \{x \in \R^2\::\: 1 - \eps < |x| <1\} \subset \R^2$ such that the angular derivatives $\frac{\partial u_i}{\partial \theta}$ of the components change sign at least $k$ times on every circle contained in $\overline B_\eps$.
\end{theo}

\begin{proof}
We apply a classical bifurcation result of Crandall and Rabinowitz, see \cite[Lemma 1.1]{crandall.rabinowitz}. Let $\tilde Y$ denote the space of functions $u \in C(\overline{B_\eps})$ which are symmetric with respect to reflection at the $x_1$-axis and $\tilde X$ the space of all $u \in \tilde Y \cap C^2(\overline {B_\eps})$ with $\partial_\nu u=0$ on $\partial B_\eps$. Then $\tilde X$ and $\tilde Y$ are Banach spaces with respect to the norms of $C^2(\overline{B_\eps})$, $C(\overline {B_\eps})$, respectively. Let $X:= \tilde X \times \tilde X$, $Y:= \tilde Y \times \tilde Y$, and let $F: (0,\infty) \times  X \to Y$ be given by 
$$
F(\lambda, u)= {\Delta u_1 + \lambda(u_1+\lambda) - (u_1+\lambda)(u_2+\lambda) 
 \choose \Delta u_2+ \lambda(u_2+\lambda) - (u_1+\lambda)(u_2+\lambda)} = {\Delta u_1 -(u_1+\lambda) u_2 
 \choose \Delta u_2 - (u_2+\lambda)u_1}
$$
Then we have $F(\lambda,0)= 0$ for all $\lambda >0$. Moreover, $u=(u_1,u_2) \in X$ solves (\ref{Lotka:Volterra:system-elliptic}) if and only if $F(\lambda, u_1-\lambda,u_2-\lambda) = 0$. We consider the partial derivative
$$
\partial_u F : (0,\infty) \times  X \to \cL(X,Y),\qquad \partial_u F(\lambda,u)v= { \Delta v_1 -u_2 v_1 - (u_1+\lambda)v_2 
 \choose \Delta v_2 - u_1v_2 - (u_2+\lambda) v_1}.
$$
For $\lambda>0$ we put 
$$
A_\lambda:= \partial_u F(\lambda,0) \in \cL(X,Y),\qquad  A_\lambda v = {\Delta v_1-\lambda v_2 
 \choose \Delta v_2 - \lambda v_1},
$$
and we let $N(A_\lambda)$ resp. $R(A_\lambda)$ denote the kernel and the image of $A_\lambda$, respectively. If $v \in N(A_\lambda)$, then $c:= v_1+v_2$ satisfies $-\Delta c +\lambda c=0$ in $B_\eps$ and $\partial_\nu c= 0$ on $\partial B_\eps$,
which easily implies that $c \equiv 0$ since $\lambda > 0$. Consequently, $v \in N(A_\lambda)$ if and only if $v_2=-v_1$ and 
$$
-\Delta v_1 = \lambda v_1 \quad \text{in $B_\eps$,} \qquad \partial_\nu v_1= 0 \quad \text{on $\partial B_\eps$.}
$$
By separation of variables, there exists $k \in {\mathbb {N}} \cup \{0\}$ such that in polar coordinates we have $v_1(r,\theta)=\varphi(r) \cos(k \theta)$,  where $\varphi \in C^2([1-\eps,1])$ satisfies 
\begin{equation}
  \label{eq:23}
 -\Delta_r \varphi +\frac{k^2}{r^2} \varphi= \lambda \varphi \quad \text{in $(1-\eps,1)$,} \qquad \partial_r \varphi(1-\eps)= \partial_r \varphi(1)=0 
\end{equation} 
with $\Delta_r= \partial_{rr} + \frac{1}{r}\partial_r$.  Let $\lambda_j(k, \eps) \ge 0$ denote the $j$-th eigenvalue of (\ref{eq:23}), counted with multiplicity in increasing order. By Sturm-Liouville theory, these eigenvalues are simple.  It is easy to see that, for fixed $k \in {\mathbb{N}} \cup \{0\}$, we have $\lambda_1(k, \eps) \to k^2$ and $\lambda_{j}(k,\eps) \to \infty$ for $j \ge 2$ as $\eps \to 0$. Moreover, $\lambda_j(k,\eps)$ is strictly increasing in $k$ for fixed $\eps>0$. We now fix $k \in {\mathbb{N}}$, and we choose $\eps= \eps(k)>0$ small enough such that $\lambda_2(0,\eps) >\lambda_1(k,\eps)$. We then set $\lambda_*= \lambda_1(k,\eps)>0$, and we let $\varphi$ denote the unique positive eigenfunction of (\ref{eq:23}) for $\lambda= \lambda_*$ with $\|\varphi\|_\infty=1$.  It then follows that $N(A_{\lambda_*})$ is spanned by $(\psi,-\psi) \in X$ with $\psi(r,\theta)= \varphi(r) \cos(k \theta)$. Moreover, it easily follows from integration by parts that $\int_{B_\eps} \psi (v_1 - v_2)\,dx = 0$ for every $v=(v_1,v_2) \in R(A_{\lambda_*})$. Since $A_{\lambda_*}$ is a Fredholm operator of index zero, we thus conclude that 
$$
R(A_{\lambda_*})= \Bigl \{v \in Y\::\: \int_{B_\eps} \psi (v_1 - v_2)\,dx = 0 \Bigr\}.
$$
In particular, since $\frac{d}{d \lambda} A_\lambda \,v = {-v_2 \choose -v_1}$ for $v \in X$ and $\lambda>0$, we find that $\frac{d}{d \lambda} A_\lambda \Big|_{\lambda=\lambda_*}(\psi,-\psi)=(-\psi,\psi) \not \in R(A_{\lambda_*})$. Hence the assumptions of \cite[Lemma 1.1]{crandall.rabinowitz}
are satisfied, and thus there exists $\delta>0$ and $C^1$-functions $\lambda: (-\delta,\delta) \to 
(0,\infty)$ and $u: (-\delta,\delta)  \to X$ such that $\lambda(0)= \lambda_*$, $F(\lambda(t),u(t))=0$ for all $t \in (-\delta,\delta)$ and $u(t)= t (\psi,-\psi) + o(t)$ in $X$. Hence, fixing $t \in (-\delta,\delta) \setminus \{0\}$ sufficiently close to zero and considering $\lambda:= \lambda(t)$, we find that 
$u=(u_1,u_2)$ with $u_1= u_1(t)+ \lambda(t)$, $u_2=u_2(t)+\lambda(t)$ is a positive solution of  
(\ref{Lotka:Volterra:system-elliptic}) such that the angular derivatives $\frac{\partial u_i}{\partial \theta}$ of the components change sign at least $k$ times on every circle contained in $\overline B_\eps$.
\end{proof}

\end{document}